\documentclass[a4paper]{article}

\usepackage[top=2.5cm,bottom=5cm,left=1.75cm,right=1.75cm]{geometry}

\usepackage{graphics}
\usepackage{bm}
\usepackage{amsmath}
\usepackage{amsfonts}
\usepackage{amsthm}
\usepackage{mathrsfs}  
\usepackage{xcolor}
\usepackage{manfnt}
\usepackage[utf8]{inputenc}
\usepackage[T1]{fontenc}
\usepackage[american]{babel}
\usepackage{booktabs}
\usepackage{graphicx}  
\usepackage{setspace}
\usepackage{epstopdf}

\newcommand{\field}{\mathbb}

\newcommand{\reals}{\field{R}}

\newcommand{\euclideans}{\field{E}}

\newcommand{\ie}{\textit{i.e.,}~}
\newcommand{\eg}{\textit{e.g.,}~}
\newcommand{\ia}{\textit{inter alia}}

\newcommand{\uii}{\bm{i}^1}
\newcommand{\ujj}{\bm{i}^2}
\newcommand{\ukk}{\bm{i}^3}

\newcommand{\zero}{\bm{0}}
\newcommand{\coor}{\bm{q}}
\newcommand{\weight}{\bm{C}}
\newcommand{\strain}{\bm{e}}
\newcommand{\stress}{\bm{s}}
\newcommand{\onehalf}{\frac{1}{2}}
\newcommand{\bxi}{\bm{\xi}}
\newcommand{\force}{\bm{f}}
\newcommand{\lagrangean}{\mathcal{L}}
\newcommand{\cost}{\mathcal{J}}
\newcommand{\cmanifold}{\mathcal{Z}}
\newcommand{\datacmanifold}{\widetilde{\cmanifold}}
\newcommand{\reccmanifold}{\check{\cmanifold}}
\newcommand{\blambda}{\bm{\lambda}}
\newcommand{\bmu}{\bm{\mu}}
\newcommand{\bnu}{\bm{\nu}}
\newcommand{\bchi}{\bm{\chi}}
\newcommand{\bphi}{\bm{\varphi}}
\newcommand{\btheta}{\bm{\theta}}
\newcommand{\pd}[1]{\partial_{#1}}
\newcommand{\ba}{\bm{a}}

\newcommand{\bg}{\bm{g}}
\newcommand{\bx}{\bm{x}}
\newcommand{\bG}{\bm{G}}
\newcommand{\bn}{\bm{n}}
\newcommand{\gl}{\bm{E}_{\flat}}
\newcommand{\rotation}{\bm{\Lambda}}
\newcommand{\KKT}{\bm{S}}
\newcommand{\refe}{\mathrm{ref}}

\newcommand{\ofcoor}{(\coor)}
\newcommand{\bgamma}{\bm{\gamma}}
\newcommand{\bomega}{\bm{\omega}}
\newcommand{\coorz}{\bm{\varphi}_0}
\newcommand{\coori}{\bm{d}_1}
\newcommand{\coorj}{\bm{d}_2}
\newcommand{\coork}{\bm{d}_3}
\newcommand{\bforce}{\bm{n}}
\newcommand{\bmoment}{\bm{m}}

\newcommand{\bI}{\bm{I}}
\newcommand{\bB}{\bm{B}}
\newcommand{\bh}{\bm{h}}

\newcommand{\bU}{\bm{U}}

\newcommand{\bV}{\bm{V}}

\newcommand{\coorihat}{\widehat{\coori}}
\newcommand{\coorjhat}{\widehat{\coorj}}
\newcommand{\coorkhat}{\widehat{\coork}}

\newcommand{\nullproj}{\bm{N}(\coor)}

\newcommand{\straincoor}{\strain(\coor)}

\newcommand{\straindata}{\widetilde{\strain}}
\newcommand{\strainrec}{\check{\strain}}
\newcommand{\iweight}{\weight^{-1}}
\newcommand{\stressdata}{\widetilde{\stress}}
\newcommand{\stressrec}{\check{\stress}}

\newcommand{\bBcoor}{\bB\ofcoor}

\newcommand{\bBcoorT}{\bB\ofcoor^T}

\newcommand\nstrain{n_{\strain}}

\newcommand\fixNLP{$\mathrm{NLP}(\straindata, \stressdata)$}

\newcommand\sdot{^T}
\newcommand\sfix{_{\mathrm{fix}}}
\newcommand\sone{}

\newcommand\rdot{{}\cdot{}}

\newcommand{\arclen}{\sigma}
\newcommand{\pdarclen}{\partial_\arclen}

\newcommand{\Boper}{\mathfrak{B}}

\newcommand{\DD}{Data-Driven}
\newcommand{\DDCD}{\DD\ Computational Dynamics}
\newcommand{\DDCM}{\DD\ Computational Mechanics}

\newcommand{\GEB}{Geometrically Exact Beam}

\let\oldforce\force
\renewcommand{\force}{\oldforce^{\textrm{ext}}}

\newcommand{\mass}{\bm{M}}
\newcommand{\timeindex}{i}
\newcommand{\sn}{_{\timeindex}}
\newcommand{\snp}{_{\timeindex+1}}
\newcommand{\snm}{_{\timeindex-1}}
\newcommand{\snpm}{_{\timeindex+\onehalf}}
\newcommand{\snmm}{_{\timeindex-\onehalf}}
\newcommand{\bBcoornmm}{\bB(\coor\snmm)}
\newcommand{\bBcoornmmT}{\bBcoornmm^T}
\newcommand{\bBcoornpm}{\bB(\coor\snpm)}
\newcommand{\bBcoornpmT}{\bBcoornpm^T}
\newcommand{\timestep}{\Delta t}

\newcommand{\bF}{\bm{F}}

\newcommand{\bFstressnpm}{\bF(\stress\snpm)}
\newcommand{\bFstressnpmT}{\bFstressnpm^T}

\newcommand{\bgcoor}{\bg(\coor)}
\newcommand{\bGcoor}{\bG(\coor)}
\newcommand{\bGcoorT}{\bGcoor^T}

\newcommand{\bgcoornp}{\bg(\coor\snp)}
\newcommand{\bGcoorn}{\bG(\coor\sn)}
\newcommand{\bGcoornT}{\bGcoorn^T}

\newcommand{\bGcoornp}{\bG(\coor\snp)}
\newcommand{\bGcoornpT}{\bGcoornp^T}

\newcommand{\straincoornpm}{\strain(\coor\snpm)}

\newcommand{\nullmatcoorn}{\bm{N}(\coor\sn)}
\newcommand{\nullmatcoornT}{\nullmatcoorn^T}

\newcommand{\pseudobalance}{\bm{f}(\coor\snp,\stress\snpm)}

\let\oldstraindata\straindata
\renewcommand{\straindata}{\oldstraindata\snpm}

\let\oldstressdata\stressdata
\renewcommand{\stressdata}{\oldstressdata\snpm}

\let\oldstrainrec\strainrec
\renewcommand{\strainrec}{\oldstrainrec\snpm}

\let\oldstressrec\stressrec
\renewcommand{\stressrec}{\oldstressrec\snpm}

\newcommand{\bhstrainstress}{\bh(\strainrec, \stressrec)}

\let\oldblambda\blambda
\renewcommand{\blambda}{\oldblambda\snpm}

\let\oldbmu\bmu
\renewcommand{\bmu}{\oldbmu\sn}

\let\oldbnu\bnu
\renewcommand{\bnu}{\oldbnu\snp}

\let\oldbxi\bxi
\renewcommand{\bxi}{\oldbxi\snpm}

\newcommand{\argvardyn}{(\coor\snp,\strain\snpm,\stress\snpm,\blambda,\bmu,\bnu;\straindata,\stressdata)}
\newcommand{\argvardynapp}{(\strainrec,\stressrec,\coor\snp,\strain\snpm,\stress\snpm,\blambda,\bmu,\bnu,\bxi)}

\newcommand{\mytime}{t}

\newcommand{\mytimed}{\mytime\snpm}

\newcommand{\area}{\mathcal{A}_0}

\begin{document}

\newtheorem{theorem}{Theorem}[section]

\theoremstyle{definition}
\newtheorem{definition}{Definition}[section]

\def\keywords{\vspace{.5em}{\noindent \textit{Keywords}:\,\relax}}
\def\endkeywords{\par}

\providecommand{\keywords}[1]{\textbf{\textit{Keywords:}} #1}

\title{A framework for \DDCD\\ based on nonlinear optimization}

\author{Cristian Guillermo Gebhardt, Marc Christian Steinbach, Dominik Schillinger, Raimund Rolfes}

\date{}

\maketitle


\begin{abstract}
\noindent In this article, we present an extension of the formulation recently developed by the authors (\textit{A Framework for \DDCM~Based on Nonlinear Optimization}, arXiv:1910.12736 [math.NA]) to the structural dynamics setting. Inspired by a structure-preserving family of variational integrators, our new formulation relies on a discrete balance equation that establishes the dynamic equilibrium. From this point of departure, we first derive an ``exact'' discrete-continuous nonlinear optimization problem that works directly with data sets. We then develop this formulation further into an ``approximate'' nonlinear optimization problem that relies on a general constitutive model. This underlying model can be identified from a data set in an offline phase. To showcase the advantages of our framework, we specialize our methodology to the case of a geometrically exact beam formulation that makes use of all elements of our approach. We investigate three numerical examples of increasing difficulty that demonstrate the excellent computational behavior of the proposed framework and motivate future research in this direction.
\end{abstract}


\section{Introduction}

\DDCM~is a new computing philosophy that enables the evolution from conventional data-free methods to modern data-rich approaches. Its underlying concept relies on the reformulation of classical boundary value problems of elasticity and inelasticity such that material models, which are calibrated from experiments, are replaced by some form of experimental material data. On the one hand, \DDCM~eliminates some modeling errors and the associated uncertainty by employing experimental data directly. On the other hand, new sources of error emerge that are associated with the measurements and with the particular measuring technology employed. At least for the moment, there is no consensus regarding which sources of error are the most severe and therefore, there is still a long way to go.

Among recent developments, two principal approaches of \DDCM~can be distinguished. On the one hand, there is a \textit{direct} one \cite{Kirchdoerfer2016,Kirchdoerfer2017}, whose methods are based on a discrete-continuous optimization problem that attempts to assign to each material point a point in the phase space that under fulfillment of the compatibility and equilibrium constraints is closest to the data set provided. Within that framework, some very interesting advances are reported, for instance, regarding nonlinear elasticity \cite{Nguyen2018}, general elasticity \cite{Conti2018}, inelasticity \cite{Eggersmann2019}, and mixed-integer quadratic optimization problems \cite{Kanno2019}. On the other hand, there is an \textit{inverse} one \cite{Ibanez2017,Ibanez2018a,Ibanez2019}, whose methods rely on an inverse approach that attempts to reconstruct from the data sets provided a constitutive manifold with a well-defined functional structure. In the context of these two families of methods, our recent work on approximate nonlinear optimization problems \cite{Gebhardt2019d} represents a \textit{hybrid} approach, targeting at a synergistic compromise that combines their strengths and mitigates some of their main weaknesses, in particular the high computational cost associated with the resolution of a discrete-continuous nonlinear optimization problem for the \textit{direct} approach and the limitation to a special functional structure that only allows the explicit definition of stresses for the \textit{inverse} approach.

\DDCD, the application of \DDCM~principles to structural dynamics problems, is currently less developed, with a few papers published so far. For instance, in \cite{Kirchdoerfer2018}, the data-driven solvers for quasistatic problems developed in \cite{Kirchdoerfer2016,Kirchdoerfer2017} were extended to dynamics, relying on variational time-stepping schemes such as the Newmark algorithm. In \cite{Gonzalez2019}, a thermodynamically consistent approach that relies on the ``General Equation for Non-Equilibrium Reversible-Irreversible Coupling'' formalism was presented, which enforces by design the conservation of energy and positive production of entropy.

The central goal of the present work is the formulation of an approximate nonlinear optimization problem for \DDCD, which can be understood as the structural dynamics counterpart of the formulation previously developed by the authors in \cite{Gebhardt2019d}. The proposed approximate nonlinear optimization problem relies on a discrete balance equation, which is inspired by a class of variational integrators \cite{Marsden2001,Lew2003,Lew2004,Kale2007,Fong2008,Leyendecker2008,Betsch2010,Lew2016} and represents the dynamic equilibrium. Since in our approach, no special functional structure of the constitutive manifold is assumed, the existence of an energy function is in general not guaranteed, and therefore, energy-momentum methods \cite{Gonzalez1996,McLachlan1999,Betsch2010,Romero2012,Gebhardt2019c} are not directly applicable. Firstly, the proposed framework improves computational efficiency and robustness with respect to the type of solvers that rely on discrete-continuous optimization problems. In particular, our approximate nonlinear optimization problem can be solved with local Sequential Quadratic Programming methods, circumventing the necessity of employing meta-heuristic methods. Secondly, the proposed framework can deal with implicitly defined stress-strain relations and kinematic constraints, thus enlarging its range of applicability. Lastly, we consider the case of a geometrically exact beam element to demonstrate the advantages of our approach. Such a finite element model makes full use of our computational machinery. Be aware that our primary goal is a proof of concept for our new approximate nonlinear optimization approach for \DDCD~and therefore, the identification of the underlying constitutive manifold is not addressed here.

The remainder of this work is organized as follows: Section 2, the core of this article, presents two optimization problems for \DDCD~that are built upon a time integration approach inspired by a class of structure-preserving methods. The first one is an ``exact'' discrete-continuous nonlinear optimization problem that works directly with data sets. Such a problem can be considered as the starting point and is not going to be solved within this work. The second one is an ``approximate'' nonlinear optimization problem that relies on a general approximation of the underlying constitutive manifold, which circumvents completely the necessity of online handling of data sets. For both problems, we define the associated Lagrangian functions and derive explicitly the first order optimality conditions as well as the corresponding KKT matrices. In section 3, we specialize the proposed methodology for the geometrically exact beam finite element in a purely dynamic setting. This particular structural model has been chosen because it makes use of all elements of our approach. Section 4 presents simulation results that illustrate the capability of the derived approach with special emphasis on preserved quantities along the discrete motion, which is seen as the solution of a sequence of successive nonlinear optimization problems. Finally, in Section 5, we draw concluding remarks and propose future work.

\section{Nonlinear optimization problems}

The definition of successive nonlinear optimization problems in \DDCM\
implies the partitioning of the considered time interval
$[\mytime_a, \mytime_b]$ into subintervals $[\mytime\sn, \mytime\snp]$
such that $\mytime_a = \mytime_0 < \dots < \mytime_N = \mytime_b$.
We consider an equidistant partitioning by a fixed time step,
\ie $\mytime\snp - \mytime\sn = \timestep$ $\forall \timeindex \in [0, N - 1]$.
A very simple scalar cost function to be minimized at time instant
$\mytimed \in (\mytime\sn, \mytime\snp)$ can be defined as
\begin{equation}
  \label{eq:cost}
  \cost(\oldstraindata\snpm, \oldstressdata\snpm, \strain\snpm, \stress\snpm)
  =
  \frac12 \|\strain\snpm - \oldstraindata\snpm\|^2_{\weight} +
  \frac12 \|\stress\snpm - \oldstressdata\snpm\|^2_{\iweight}
  ,
\end{equation}
where the pair $(\strain\snpm, \stress\snpm) \in \cmanifold$
denotes continuous strain and stress variables
$\strain\snpm$ and $\stress\snpm$, respectively,
a given finite data set $\datacmanifold$
contains strain and stress measurements
$(\oldstraindata\snpm, \oldstressdata\snpm)$,
$\weight \in \reals^{\nstrain \times \nstrain}$
is a symmetric positive-definite weight matrix with inverse $\iweight$,
and $\|\cdot\|_{\weight}$ and $\|\cdot\|_{\iweight}$
are norms derived from an inner product.
At this point, there is no necessity to specify $\nstrain$
since it depends on the structural model considered,
which for now remains unspecified.
The cost function \eqref{eq:cost} has to be minimized
under the following constraints:
\textit{i}) the compatibility equation that enforces the equivalence
between strain variables and displacement-based strains
at time instant $\mytime\snpm$,
\begin{equation}
  \strain\snpm - \straincoornpm = \zero,
\end{equation}
in which $\coor\snpm \in Q \subset \reals^{m+n}$
is the vector of generalized coordinates
and $Q$ stands for the configuration manifold;
\textit{ii}) the discrete balance equation
that establishes the dynamic equilibrium,
for instance, we chose an approximation
inspired by a family of variational integrators
\cite{Marsden2001,Lew2003,Lew2004,Kale2007,Fong2008,Leyendecker2008,Betsch2010,Lew2016}
that renders the dynamic equilibrium at time instant $\mytime\sn$ as
\begin{equation}
  \label{eq:equilibrium}
  \mass \frac{\coor\snp - 2 \coor\sn + \coor\snm}{\timestep} +
  \frac{\timestep}{2} \bigl(
    \bBcoornmmT \stress\snmm + \bBcoornpmT\stress\snpm
  \bigr) + \timestep\bGcoornT \bchi\sn -
  \frac{\timestep}{2} \bigl( \force\snmm + \force\snpm \bigr) = \zero,
\end{equation}
in which $\mass \in T_{\coor}^*Q \times T_{\coor}^*Q$
represents the constant mass matrix,
\begin{equation}
  \coor\snmm \approx \frac{\coor\snm + \coor\sn}{2}
  \qquad\textrm{and}\qquad
  \coor\snpm \approx \frac{\coor\sn + \coor\snp}{2},
\end{equation}
$\bBcoornpm = \pd{\coor\snpm} \strain(\coor\snpm)
\in L(T_{\coor\snpm} Q, \reals^{n_e})$
is the Jacobian matrix of the displacement-based strains,
similarly for $\bBcoornmm$,
$\bGcoorn = \pd{\coor\sn} \bg(\coor\sn) \in L(T_{\coor\sn} Q, \reals^m)$
is the Jacobian matrix of the kinematic constraints
at time instant~$\mytime\sn$,
$\bchi\sn \in \reals^m$ is the corresponding vector of Lagrange multipliers,
and $\force\snpm\in T_{\coor\snpm}^*Q$ represents the vector
of generalized external loads, similarly for $\force\snmm$;
and, \textit{iii}) the kinematic constraints at time instant $\mytime\snp$,
\begin{equation}
  \bgcoornp = \zero,
\end{equation}
a finite set of integrable restrictions that belongs to $\reals^m$.
As usual in the finite element setting,
we assume that the Jacobian matrix of the displacement-based strains $\bBcoor$
and the Jacobian matrix of the constraints $\bG(\coor)$ are linear in $\coor$,
yielding substantial simplifications when calculating higher-order derivatives.
Since the existence of an energy function is in general not guaranteed,
energy-momentum methods are not directly applicable
\cite{Gonzalez1996,McLachlan1999,Betsch2010,Romero2012,Gebhardt2019c}.

\pagebreak[1]

Now, to brief\/ly investigate the conservation properties
of the adopted time integration scheme,
let us first neglect the external forces and define the discrete momenta:
\begin{subequations}
  \begin{align}
    \bm{p}^{+}_{\timeindex}
    &=
    \mass \frac{\coor\sn - \coor\snm}{\timestep} -
    \frac{\timestep}{2} \bBcoornmmT \stress\snmm -
    \frac{\timestep}{2} \bGcoornT \bchi\sn \in T^{*}_{\coor\sn}Q, \\
    \bm{p}^{-}_{\timeindex}
    &=
    \mass \frac{\coor\snp - \coor\sn}{\timestep} +
    \frac{\timestep}{2} \bBcoornpmT \stress\snpm +
    \frac{\timestep}{2} \bGcoornT \bchi\sn \in T^{*}_{\coor\sn}Q.
  \end{align}
\end{subequations}
These definitions are inspired by the discrete Legendre transforms
that are widely used in the context of variational integrators.
Having at hand the discrete momenta,
the discrete balance equation can be rewritten as
$\bm{p}^{+}_{\timeindex} - \bm{p}^{-}_{\timeindex} = \zero$,
which leads to the existence of a unique momentum
$\bm{p}\sn$ at time instant $\mytime\sn$.
In this discrete setting,
there are two possible definitions of linear momentum, namely
\begin{equation}
  \bm{l}^{-}_{\timeindex}
  =
  \sum_{a=1}^{N_{\textrm{nodes}}}
  \mathfrak{L}_{\textrm{trans}} (\bm{p}^{-}_{\timeindex})_a
  \quad\textrm{and}\quad
  \bm{l}^{+}_{\timeindex}
  =
  \sum_{a=1}^{N_{\textrm{nodes}}}
  \mathfrak{L}_{\textrm{trans}} (\bm{p}^{+}_{\timeindex})_a
  ,
\end{equation}
in which $\mathfrak{L}_{\textrm{trans}}$
filters out all non-translational contributions
and $N_{\textrm{nodes}}$ denotes the number of nodes.
Provided that the system under consideration is invariant under translations,
\ie the orthogonality between the internal forces
and the infinitesimal generator of translation is given,
a unique discrete linear momentum does exist, namely
$\bm{l}_d(\coor\sn, \coor\snp) =
\bm{l}^{-}_{\timeindex} = \bm{l}^{+}_{\timeindex}$,
and is an invariant of the discrete motion,
whose conservation law reads
\begin{equation}
  \bm{l}_d(\coor\sn, \coor\snp) - \bm{l}_d(\coor\snm, \coor\sn) = \zero.
\end{equation}
Likewise, there are two possible definitions of angular momentum, namely
\begin{equation}
  \bm{j}^{-}_{\timeindex}
  =
  \sum_{a=1}^{N_{\textrm{nodes}}}
  (\coor\sn)_a \times (\bm{p}^{-}_{\timeindex})_a
  \quad\textrm{and}\quad
  \bm{j}^{+}_{\timeindex}
  =
  \sum_{a=1}^{N_{\textrm{nodes}}}
  (\coor\snp)_a \times (\bm{p}^{+}_{\timeindex})_a.
\end{equation}
Similarly, provided that the system under consideration
is invariant under rotations,
\ie the orthogonality between the internal forces
and the infinitesimal generator of the rotation is given,
a unique discrete angular momentum does exist, namely
$\bm{j}_d(\coor\sn, \coor\snp) =
\bm{j}^{-}_{\timeindex} = \bm{j}^{+}_{\timeindex}$,
and is an invariant of the discrete motion,
whose conservation law reads
\begin{equation}
  \bm{j}_d(\coor\sn, \coor\snp) - \bm{j}_d(\coor\snm, \coor\sn) = \zero.
\end{equation}
Lastly, to avoid problems
caused by overdetermination and singular KKT matrices
in the subsequent optimization problems,
we eliminate the Lagrange multipliers
from \eqref{eq:equilibrium}
by means of the null-space method.
This requires a null-space basis matrix
$\nullmatcoorn \in L(\reals^n, T_{\coor\sn} Q)$ for
$\ker(\bG(\coor\sn)) =
\{ \bn \in T_{\coor\sn} Q \mid \bG(\coor\sn) \bn = \zero \in \reals^m \}$
with $n = \dim(Q) - m$, the system's number of degrees of freedom,
and $\mathrm{rank}(\nullmatcoorn) = \dim(\ker(\bG(\coor\sn))) = n$,
such that
\begin{equation}
  \bG(\coor\sn) \nullmatcoorn = \zero.
\end{equation}
Then, dynamic equilibrium adopts the form
\begin{subequations}
  \begin{align}
    \zero
    & =
    \nullmatcoorn \sdot \Bigl(
      \mass \frac{\coor\snp - 2 \coor\sn + \coor\snm}{\timestep} +
      \frac{\timestep}{2}
      \bigl( \bBcoornmmT \stress\snmm + \bBcoornpmT \stress\snpm \bigr) -
      \frac{\timestep}{2} \bigl( \force\snmm + \force\snpm \bigr)
    \Bigr) \\
    &= \pseudobalance,
  \end{align}
\end{subequations}
where only the dependency on unknown quantities
is explicitly indicated in $\pseudobalance$.

\subsection{The ``exact'' discrete-continuous nonlinear optimization problem}

\begin{definition}[Exact DCNLP]
  Employing directly the strain and stress measurements,
  each successive ``\DDCD'' problem can be defined as a
  discrete-continuous nonlinear optimization problem that can be stated as
  \begin{equation}
    \begin{split}
      \min_{(\straindata, \stressdata, \coor\snp, \strain\snpm, \stress\snpm)}
      \quad &
      \frac12 \|\strain\snpm - \straindata\|^2_{\weight} +
      \frac12 \|\stress\snpm - \stressdata\|^2_{\iweight} \\
      \mathrm{subject~to} \quad \quad \quad 
      &\strain\snpm - \straincoornpm = \zero, \\
      &\pseudobalance = \zero, \\
      &\bgcoornp = \zero.
    \end{split}
  \end{equation}
\end{definition}
Notice that the discrete variables
$(\straindata, \stressdata) \in \datacmanifold$
at time instant $t\snpm$ appear only in the cost function.
For fixed $(\straindata, \stressdata)$,
the exact DCNLP becomes a smooth nonlinear optimization problem (NLP),
referred to as \fixNLP.
Any solution provides a set of values $(\coor\snp, \strain\snpm, \stress\snpm)$
that (locally) minimizes the cost function
for the fixed data point under the constraints given above.

\begin{theorem}
  The first-order optimality conditions of \fixNLP~are:
  \begin{subequations}
    \label{eq:KKT-cond-fixNLP}
    \begin{align}
      &\delta\coor\snp:
      &-\onehalf \bBcoornpmT \blambda + \bFstressnpmT \bmu +
      \bGcoornpT \bnu &= \zero, \\
      &\delta\strain\snpm:
      &\weight (\strain\snpm - \straindata) + \blambda &= \zero, \\
      &\delta\stress\snpm:
      &\iweight (\stress\snpm - \stressdata) +
      \frac{\timestep}{2} \bBcoornpm \nullmatcoorn \bmu &= \zero, \\
      &\delta\blambda: & \strain\snpm - \straincoornpm &= \zero, \\
      &\delta\bmu: & \pseudobalance &= \zero, \\
      &\delta\bnu: & \bgcoornp &= \zero,
    \end{align}
  \end{subequations}
  where we define
  \begin{equation}
    \bFstressnpm := \pd{\coor\snp} \pseudobalance =
    \nullmatcoornT \Bigl( \frac{1}{\timestep} \mass +
    \frac{\timestep}{4} \bU_2 (\stress\snpm) \Bigr)
    \quad 
  \end{equation}
  with
  \begin{equation}
    \bU_2(\stress) := \pd{\coor}(\bBcoorT \stress) = \bU_2(\stress)^T.
  \end{equation}
\end{theorem}

\begin{proof}
  The Lagrangian function of \fixNLP~is
  \begin{equation}
    \begin{split}
      \lagrangean\sfix\argvardyn &=
      \frac12 \|\strain\snpm - \straindata\|^2_{\weight} +
      \frac12 \|\stress\snpm - \stressdata\|^2_{\iweight} \\
      &+ \blambda \sdot \bigl( \strain\snpm - \straincoornpm \bigr) \\
      &+ \bmu \sdot \pseudobalance \\
      &+ \bnu \sdot \bgcoornp,
    \end{split}
  \end{equation}
  where $\blambda \in \reals^{\nstrain}$ are Lagrange multipliers
  of the compatibility equation at time instant $t\snpm$,
  $\bmu \in \reals^n$ are Lagrange multipliers
  of the balance equation premultiplied by the null-space basis matrix
  evaluated at time instant $t\sn$, and
  $\bnu \in \reals^m$ are Lagrange multipliers of kinematic constraints
  at the instant $t\snp$.
  To derive the corresponding first-order optimality conditions,
  we calculate the variation of $\lagrangean\sfix$ as
  \begin{equation}
    \delta \lagrangean\sfix\argvardyn =
    \pd{\bx\sfix} \lagrangean\sfix(\bx\sfix; \straindata, \stressdata)
    \delta\bx\sfix
  \end{equation}
  with the primal-dual NLP variable vector
  \begin{equation}
    \bx\sfix :=
    (\coor\snp^T, \strain\snpm^T, \stress\snpm^T, \blambda^T, \bmu^T, \bnu^T)^T,
  \end{equation}
  obtaining
  \begin{equation}
    \begin{split}
      \pd{\bx\sfix} \lagrangean\sfix(\bx\sfix; \straindata, \stressdata)
      \delta\bx\sfix
      &= \delta\coor\snp \sdot \Bigl(
        -\onehalf \bBcoornpmT \blambda +
        \bigl( \pd{\coor\snp} \pseudobalance \bigr)^T \bmu +
        \bGcoornp^T\bnu \Bigr) \\
      &+ \delta\strain\snpm \sdot
      \bigl( \weight (\strain\snpm - \straindata) + \blambda \bigr) \\
      &+ \delta\stress\snpm \sdot \Bigl(
        \iweight (\stress\snpm - \stressdata) +
        \frac{\timestep}{2} \bBcoornpm \nullmatcoorn \bmu \Bigr) \\
      &+ \delta\blambda \sdot \bigl( \strain\snpm - \straincoornpm \bigr) \\
      &+ \delta\bmu \sdot \pseudobalance \\
      &+ \delta\bnu \sdot \bgcoornp.
    \end{split}
  \end{equation}
  Setting this to zero for any choice of the varied quantities
  yields the KKT conditions \eqref{eq:KKT-cond-fixNLP}.
\end{proof}

Notice that $\strain\snpm$ and $\blambda$ can be eliminated by substitution,
but as we are interested in the problem's global format,
we are not going to eliminate anything unless strictly necessary.

The linearization of the variation of $\lagrangean\sfix$ reads
\begin{equation}
  \Delta \delta\lagrangean\sfix(\bx\sfix; \straindata, \stressdata)
  =
  \delta\bx\sfix \sdot
  \KKT\sfix(\coor\snp, \stress\snpm, \blambda, \bmu, \bnu)
  \Delta\bx\sfix,
\end{equation}
where the KKT matrix $\KKT\sfix$ is symmetric indefinite and can be written
\begin{equation}
  \begin{gathered}
    \KKT\sfix(\coor\snp, \stress\snpm, \blambda, \bmu, \bnu) = \\
    \begin{bmatrix}
      -\frac14 \bU_2(\blambda) + \bV(\bnu) & \zero &
      \frac{\timestep}{4} \bU_1(\nullmatcoorn\bmu)^T &
      -\onehalf \bBcoornpmT & \bFstressnpmT & \bGcoornpT \\
      \zero & \weight & \zero & \bI & \zero & \zero \\
      \frac{\timestep}{4} \bU_1(\nullmatcoorn\bmu) & \zero & \iweight &
      \zero & \frac{\timestep}{2} \bBcoornpm \nullmatcoorn & \zero \\
      -\onehalf \bBcoornpm & \bI & \zero & \zero & \zero & \zero \\
      \bFstressnpm & \zero & \frac{\timestep}{2} \nullmatcoornT \bBcoornpmT &
      \zero & \zero & \zero \\
      \bGcoornp & \zero & \zero & \zero & \zero & \zero
    \end{bmatrix}
  \end{gathered}
\end{equation}
with
\begin{equation}
  \bV(\bm{\nu}) := \pd{\coor}(\bGcoorT\bm{\nu})=\bV(\bm{\nu})^T\quad\textrm{and}\quad\bU_1(\ba) := \pd{\coor}(\bBcoor \ba)
\end{equation}
for any constant vector $\ba \in \reals^{n+m}$.
As the KKT matrix $\KKT\sfix$ is non-singular,
all local minima are strict minima
and \fixNLP\ can be solved by local Sequential Quadratic
Programming methods.

The overall DCNLP can be treated by meta-heuristic methods.
Since it has no useful structure
with respect to the discrete variables
$(\straindata, \stressdata) \in \datacmanifold$,
a mathematically rigorous solution requires enumeration,
that is, finding the minimal value over all measurements
$(\straindata, \stressdata) \in \datacmanifold$
by solving every \fixNLP\ globally.
Therefore we suggest a different approach:
we propose to add suitable structure that enables us
to replace the DCNLP with a single approximating NLP,
as already done in the static case \cite{Gebhardt2019d}.

\subsection{The ``approximate'' nonlinear optimization problem}

The idea here is to replace the measurement data set $\datacmanifold$
by enforcing the state to belong to a reconstructed constitutive manifold
that has a precise mathematical structure
and that is derived from the data set.
The underlying assumption is, of course, that such a constitutive manifold exists
and that we can reconstruct a (smooth) implicit representation $\bh$.
The reconstructed constitutive manifold (an approximation) will enormously
facilitate the task of the data-driven solver,
avoiding the cost of solving a DCNLP,
either by enumeration,
or by heuristic or meta-heuristic methods
which can in general only provide approximate solutions
that strongly depend on the initial guess and whose convergence properties
are inferior when compared to gradient-based methods.
\begin{definition}
  An ``approximate'' constitutive manifold is defined as
  \begin{equation}
    \reccmanifold
    :=
    \{(\strainrec, \stressrec) \in \reals^{2 \nstrain} \mid
    \bhstrainstress = \zero \in \reals^{\nstrain}\}
    .
  \end{equation}
  It satisfies
  \begin{equation}
    \| \bh(\straindata, \stressdata) \| \le \varepsilon
    \quad \forall \, (\straindata, \stressdata) \in \datacmanifold
  \end{equation}
  for some accuracy $\varepsilon > 0$.
  Additionally, physical consistency requires that
  $\bh(\strainrec, \zero) = \zero$ implies $\strainrec = \zero$ and
  $\bh(\zero, \stressrec) = \zero$ implies $\stressrec = \zero$.
\end{definition}
A constitutive manifold is said to be thermomechanically consistent
if it is derived from an energy function $\Psi$
such that the following functional structure holds
\cite{Crespo2017,Ibanez2017,Gebhardt2019d}:
\begin{equation}
  \bhstrainstress = \stressrec - \pd{\strainrec} \Psi(\strainrec).
\end{equation}
However, in the case of new composite materials
or metamaterials that exhibit non-convex responses,
the reconstruction of the energy function may not be very convenient.
More importantly, in some cases the formulation
of an energy function may not even be possible.
Thus, we adopt the constitutive manifold $\reccmanifold$
as introduced previously without assuming any special functional structure
of the constitutive constraint $\bh$.
Further specializations are possible and should be instantiated
for specific applications of the proposed formulation.

\begin{definition}[Approximate NLP]
  Each successive ``\DDCD'' problem can be approximated
  as a nonlinear optimization problem of the form
  \begin{equation}
    \begin{split}
      \min_{(\strainrec, \stressrec, \coor\snp, \strain\snpm, \stress\snpm)}
      \quad &
      \frac12 \|\strain\snpm - \strainrec\|^2_{\weight} +
      \frac12 \|\stress\snpm - \stressrec\|^2_{\iweight} \\
      \mathrm{subject~to} \quad \quad \quad
      &\strain\snpm - \straincoornpm = \zero, \\
      &\pseudobalance = \zero, \\
      &\bgcoornp = \zero, \\
      &\bhstrainstress = \zero.
    \end{split}
  \end{equation}
\end{definition}

\begin{theorem}
  The first-order optimality conditions of the approximate NLP are:
  \begin{subequations}
    \label{eq:KKT-cond-aNLP}
    \begin{align}
      &\delta\strainrec:
      &-\weight (\strain\snpm - \strainrec) +
      \bigl( \pd{\strainrec} \bhstrainstress \bigr)^T \bxi &= \zero, \\
      &\delta\stressrec:
      &-\iweight (\stress\snpm - \stressrec) +
      \bigl( \pd{\stressrec} \bhstrainstress \bigr)^T \bxi &= \zero, \\
      &\delta\coor\snp:
      &-\onehalf \bBcoornpmT \blambda +
      \bFstressnpmT \bmu + \bGcoornpT \bnu &= \zero,\\
      &\delta\strain\snpm:
      &\weight (\strain\snpm - \straindata) + \blambda &= \zero, \\
      &\delta\stress\snpm:
      &\iweight (\stress\snpm - \stressdata) +
      \frac{\timestep}{2} \bBcoornpm \nullmatcoorn \bmu &= \zero, \\
      &\delta\blambda: &\strain\snpm - \straincoornpm &= \zero, \\
      &\delta\bmu: &\pseudobalance &= \zero, \\
      &\delta\bnu: &\bgcoornp &= \zero,\\
      &\delta\bxi: &\bhstrainstress &= \zero.
    \end{align}
  \end{subequations}
\end{theorem}

\begin{proof}
  The Lagrangian function of the approximate NLP is given by
  \begin{equation}
    \begin{split}
      \lagrangean\sone\argvardynapp
      &=
      \lagrangean\sfix(\bx\sfix; \strainrec, \stressrec) +
      \bxi \sdot \bhstrainstress \\
      &=
      \frac12 \|\strain\snpm - \strainrec\|^2_{\weight} +
      \frac12 \|\stress\snpm - \stressrec\|^2_{\iweight} \\
      &+ \blambda \sdot \bigl( \strain\snpm - \straincoornpm \bigr) \\
      &+ \bmu \sdot \pseudobalance \\
      &+ \bnu \sdot \bgcoornp \\
      &+ \bxi \sdot \bhstrainstress
    \end{split}
  \end{equation}
  where $\bxi \in \reals^{\nstrain}$ are Lagrange multipliers
  that correspond to the enforcement of the strain and stress states
  to remain on the constitutive manifold.

  To find the first-order optimality conditions,
  the variation of $\lagrangean\sone$ is calculated as
  \begin{equation}
    \begin{split}
      \pd{\bx\sone} \lagrangean\sone(\bx\sone) \delta\bx\sone
      &=
      \delta\strainrec \sdot \bigl(
        \weight (\strainrec - \strain\snpm) +
        \bigl( \pd{\strainrec} \bhstrainstress \bigr)^T \bxi \bigr) \\
      &+ \delta\stressrec \sdot \bigl(
        \iweight (\stressrec - \stress\snpm) +
        \bigl( \pd{\stressrec} \bhstrainstress \bigr)^T \bxi \bigr) \\
      &+ \delta\coor\snp \sdot \Bigl(
        -\onehalf \bBcoornpmT \blambda + \bFstressnpmT \bmu + \bGcoornpT \bnu
        \Bigr) \\
      & +\delta\strain\snpm \sdot
      \bigl( \weight (\strain\snpm - \straindata) + \blambda \bigr) \\
      & +\delta\stress\snpm \sdot \Bigl(
        \iweight (\stress\snpm - \stressdata) +
        \frac{\timestep}{2} \bBcoornpm \nullmatcoorn \bmu \Bigr) \\
      &+ \delta\blambda \sdot \bigl( \strain\snpm - \straincoornpm \bigr) \\
      &+ \delta\bmu \sdot \pseudobalance \\
      &+ \delta\bnu \sdot \bgcoornp \\
      &+ \delta\bxi \sdot \bhstrainstress
    \end{split}
  \end{equation}
  with
  \begin{equation}
    \bx\sone := (\strainrec^T, \stressrec^T, \bx\sfix^T, \bxi^T)^T.
  \end{equation}
  Setting this to zero for any choice of the varied quantities
  yields the KKT conditions \eqref{eq:KKT-cond-aNLP}.
\end{proof}

The linearization of the variation of $\lagrangean\sone$ can be expressed as
\begin{equation}
  \Delta\delta\lagrangean\sone(\bx)
  =
  \delta\bx\sone \sdot
  \KKT\sone
  (\strainrec, \stressrec, \coor\snp, \stress\snpm, \blambda, \bmu, \bnu, \bxi)
  \Delta\bx\sone.
\end{equation}
Here the KKT matrix $\KKT\sone$ can be written as
\begin{equation}
  \begin{gathered}
    \KKT\sone
    (\strainrec, \stressrec, \coor\snp, \stress\snpm, \blambda, \bmu, \bnu, \bxi)
    = \\
    \begin{bmatrix}
      \KKT_{\oldstrainrec \oldstrainrec}(\strainrec, \stressrec, \bxi) &
      \KKT_{\oldstressrec \oldstrainrec}(\strainrec, \stressrec, \bxi)^T &
      \KKT_{\bx \oldstrainrec}^T &
      \pd{\strainrec} \bg(\strainrec,\stressrec)^T \\
      \KKT_{\oldstressrec \oldstrainrec}(\strainrec, \stressrec, \bxi) &
      \KKT_{\oldstressrec \oldstressrec}(\strainrec, \stressrec, \bxi) &
      \KKT_{\bx\sfix\oldstressrec}^T &
      \pd{\stressrec} \bg(\strainrec,\stressrec)^T \\
      \KKT_{\bx\sfix\oldstrainrec} &
      \KKT_{\bx\sfix\oldstressrec} &
      \KKT\sfix(\coor\snp, \stress\snpm, \blambda, \bmu, \bnu) &
      \zero \\
      \pd{\strainrec} \bg(\strainrec, \stressrec) &
      \pd{\stressrec} \bg(\strainrec, \stressrec) &
      \zero &
      \zero \\
    \end{bmatrix}
  \end{gathered}
\end{equation}
with
\begin{subequations}
  \begin{align}
    \KKT_{\oldstrainrec \oldstrainrec}(\strainrec, \stressrec, \bxi)
    &:=
    \weight + \pd{\strainrec}
    \bigl( \pd{\strainrec} \bg(\strainrec, \stressrec)^T \bxi \bigr) =
    \KKT_{\oldstrainrec \oldstrainrec}(\strainrec, \stressrec, \bxi)^T, \\
    \KKT_{\oldstressrec \oldstrainrec}(\strainrec, \stressrec, \bxi)
    &:=
    \pd{\stressrec}
    \bigl( \pd{\strainrec} \bg(\strainrec, \stressrec)^T \bxi \bigr), \\
    \KKT_{\oldstressrec \oldstressrec}(\strainrec, \stressrec, \bxi)
    &:=
    \iweight + \pd{\stressrec}
    \bigl( \pd{\stressrec} \bg(\strainrec, \stressrec)^T \bxi \bigr) =
    \KKT_{\oldstressrec \oldstressrec}(\strainrec, \stressrec, \bxi)^T, \\
    \KKT_{\bx\sfix \oldstrainrec}
    &:=
    \begin{bmatrix} \zero & -\weight & \zero & \zero & \zero \end{bmatrix}^T, \\
    \KKT_{\bx\sfix \oldstressrec}
    &:=
    \begin{bmatrix} \zero & \zero & -\iweight & \zero & \zero \end{bmatrix}^T.
  \end{align}
\end{subequations}
Again the KKT matrix $\KKT\sone$ is non-singular,
hence all local minima are strict and the approximate NLP is well solvable.


\section{Specialization of the proposed approach}

In this section, we describe a structural model
that is reformulated within the proposed setting of \DDCD.
The model is a \DD~\GEB~that is given in a
frame-invariant path-independent finite element formulation.
It relies on a kinematically constrained approach,
where the orientation of the cross section is described by means of
three vectors that are constrained to be mutually orthonormal.
This example has a very favorable mathematical structure
that is exploited to derive analytically all the ingredients related to the finite element formalism.

\subsection{\DD~\GEB (dynamic setting)}

The position of any point belonging to the beam shown in Figure 1 can be written as
\begin{equation}
  \bphi(\btheta)
  =
  \coorz(\theta^3) + \theta^1 \coori(\theta^3) + \theta^2 \coorj(\theta^3)\in \reals^3,
\end{equation}
in which $\coorz \in \reals^3$ is the position vector of the beam axis
and $\coori \in S^2$, $\coorj \in S^2$ together with $\coork \in S^2$
are three mutually orthonormal directors.
The directors can be described by means of the unit sphere,
which is a nonlinear, smooth, compact, two-dimensional manifold
that can be embedded in $\reals^3$ as
\begin{equation}
  \label{eq-stwo}
  S^{2} := \{\bm{d} \in \reals^3 \mid \bm{d} \sdot \bm{d} = 1\}.
\end{equation}
Special attention must be paid to the fact that this
manifold possesses no special algebraic structure,
specifically group-like structure \cite{Eisenberg1979}.
On that basis, the rotation tensor for the cross section is simply obtained as
$\rotation = \coori\otimes\uii+\coorj\otimes\ujj+\coork\otimes\ukk \in SO(3)$,
in which $\{\uii, \ujj, \ukk\}$ is the dual basis
of the ambient space $\euclideans^3$
($\reals^3$ with the standard Euclidean structure),
\ie the basis of the space of row vectors.
The group of rotations is a nonlinear, smooth, compact,
three-dimensional manifold defined as
\begin{equation}
  SO(3) :=
  \{
  \rotation \in \reals^{3 \times 3} \mid
  \rotation^T \rotation = \bI,
  \det\rotation = 1
  \}.
\end{equation}
In contrast to the unit sphere,
this manifold does possess a group-like structure
when considered with the tensor multiplication operation,
hence it is a Lie group.

\begin{figure}[tp]
  \centering
  \includegraphics[width=0.5\textwidth]{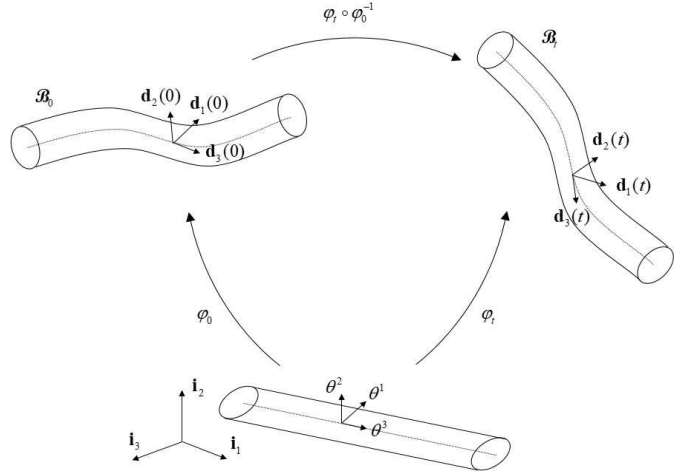}
  \caption{The geometrically exact beam: evolution among configurations through the regular motion $\bphi(\rdot; t) \circ \bphi(\rdot;0)^{-1}$.}
  \label{fig:beam}
\end{figure}

The set of parameters $\btheta = (\theta^1, \theta^2, \theta^3)$
is chosen in a way that the vector
$\bar{\btheta} = \theta^1 \coori + \theta^2 \coorj$
completely describes the cross section.
In the context of geometrically exact beams,
the doubly-covariant Green-Lagrange strain tensor, $\gl(\bphi)$,
can be simplified by eliminating quadratic strains.
Thus, its components are approximated as
\begin{equation}
  E_{ij}
  \approx
  \mathrm{symm}(\delta_{i3}\delta_{jk}((\gamma^k
  -\gamma^k_{\refe}
  )-\epsilon^k_{lm}\bar{\theta}^l(\omega^m
  -\omega^m_{\refe}
  ))),
\end{equation}
where $\mathrm{symm}(\cdot)$ stands for the symmetrization
of the tensor considered.
The subindex ``$\refe$'' indicates the stress-free configuration,
$\delta_{ij}$ denotes the Kronecker delta,
and $\epsilon^i_{jk}$ is the alternating symbol
that appears in the computation of the cross product
in three-dimensional Euclidean space.
From now on, we set $\theta^3 = \arclen$ to indicate every reference
related to the arc length of the beam.
The scalars $\gamma^i$ are the components of a first deformation vector defined as
\begin{equation}
  \bgamma
  =
  \begin{pmatrix}
    \coori \sdot \pd{\arclen} \coorz \\
    \coorj \sdot \pd{\arclen} \coorz \\
    \coork \sdot \pd{\arclen} \coorz
  \end{pmatrix}
  .
\end{equation}
For shear refer to first and second components,
and for elongation refer to the third one.
The scalars $\omega^i$ are the components of a second deformation vector defined as
\begin{equation}
  \bomega
  =
  \frac12
  \begin{pmatrix}
    \coork \sdot \pd{\arclen} \coorj - \coorj \sdot \pd{\arclen} \coork \\
    \coori \sdot \pd{\arclen} \coork - \coork \sdot \pd{\arclen} \coori \\
    \coorj \sdot \pd{\arclen} \coori - \coori \sdot \pd{\arclen} \coorj
  \end{pmatrix}
  .
\end{equation}
For bending refer to first and second components,
and for torsion refer to the third one.
For sake of compactness, let us introduce the vector
containing all kinematic fields,
\begin{equation}
  \coor(\arclen) = (\coorz(\arclen)^T, \coori(\arclen)^T, \coorj(\arclen)^T, \coork(\arclen)^T)^T,
\end{equation}
and the vector that gathers the two strain measures
obtained from the kinematic field,
\begin{equation}
  \straincoor =
  \begin{pmatrix}
    \bgamma\ofcoor - \bgamma_{\refe} \\
    \bomega\ofcoor - \bomega_{\refe}
  \end{pmatrix}
  .
\end{equation}
Additionally, we need to introduce the vector containing
all generalized ``strain'' fields that is going to be tied
by the compatibility equation,
\begin{equation}
  \strain =
  \begin{pmatrix}
    \bgamma - \bgamma_{\refe} \\
    \bomega - \bomega_{\refe}
  \end{pmatrix}
  ,
\end{equation}
and the vector containing the two generalized ``stress'' fields,
\begin{equation}
  \stress =
  \begin{pmatrix}
    \bforce \\
    \bmoment
  \end{pmatrix}
  ,
\end{equation}
which contains the cross sectional force and moment resultants,
\ie three force components and three moment components.

The operator that relates the variation of the displacement-based strains
to the variation of the kinematic fields through the relation
$\delta \straincoor = \Boper(\coor) \delta \coor$ has the explicit form
\begin{equation}
  \Boper(\coor) = \frac12
  \begin{bmatrix}
    2 \coori^T \pdarclen & 2 \pd{\arclen} \coorz^T & \zero & \zero \\
    2 \coorj^T \pdarclen & \zero & 2 \pd{\arclen} \coorz^T & \zero \\
    2 \coork^T \pdarclen & \zero & \zero & 2 \pd{\arclen} \coorz^T \\
    \zero & \zero & \coork^T \pdarclen - \pd{\arclen} \coork^T &
    \pd{\arclen} \coorj^T - \coorj^T \pdarclen \\
    \zero & \pd{\arclen} \coork^T - \coork^T \pdarclen &
    \zero & \coori^T \pdarclen - \pd{\arclen} \coori^T \\
    \zero & \coorj^T \pdarclen - \pd{\arclen} \coorj^T &
    \pd{\arclen} \coori^T - \coori^T \pdarclen & \zero
  \end{bmatrix}
  .
\end{equation}

The mass matrix per unit of length is given by
\begin{equation}
\mass =
\begin{bmatrix}
{} \mathscr{E}_{00}\bI & \mathscr{E}_{01}\bI & \mathscr{E}_{02}\bI & \zero \\
{} \mathscr{E}_{01}\bI & \mathscr{E}_{11}\bI & \mathscr{E}_{12}\bI & \zero \\
{} \mathscr{E}_{02}\bI & \mathscr{E}_{12}\bI & \mathscr{E}_{22}\bI & \zero \\
{}               \zero &               \zero &               \zero & \zero \\
\end{bmatrix}
,
\end{equation}
where $\bI$ is the $3\times3$-identity matrix and $\mathscr{E}_{ij}$ is computed by means of $\int_{\area}\varrho_0\theta^i\theta^j\mathrm{d}\area$ for $i$ and $j$ from $0$ to $2$, with $\varrho_0$ and  $\area$ representing the mass density per unit volume and the cross sectional area, both at the reference configuration, correspondingly.

To perform the spatial discretization of the geometrically exact beam
into two-node finite elements,
we approximate the kinematic fields as well as their
admissible variations with first-order Lagrangian functions. Upon such a discretization, we have that  $\Boper\leadsto\bB$.
The adopted numerical quadrature for the integration
of elemental contributions is the standard Gauss-Legendre quadrature rule.
As usual, the integrals involving internal terms are computed by means
of a one-point integration scheme that avoids shear locking issues.
Therefore, the evaluation of the kinematic fields at the single Gauss point
is in fact an average of the nodal values,
and their derivatives with respect to the arc length turn out to be
the simplest directed difference of the nodal values.
Moreover, even for coarse discretizations,
no additional residual stress corrections are necessary.
For an extensive treatment of geometrically exact beams
in the non-data-driven finite element setting, see
\cite{Romero2002, Romero2004, Betsch2002, Gebhardt2019a, Gebhardt2019b}.

Finally and as in \cite{Betsch2002}, the mutual orthonormality condition among the directors
is enforced at the nodal level by means of the internal constraint
\begin{equation}
  \bgcoor =
  \frac12 \begin{pmatrix}
    \coori \cdot \coori - 1 \\
    \coorj \cdot \coorj - 1 \\
    \coork \cdot \coork - 1 \\
    2 \coorj \cdot \coork \\
    2 \coori \cdot \coork \\
    2 \coori \cdot \coorj
  \end{pmatrix}
  .
\end{equation}
The associated Jacobian matrix is
\begin{equation}
  \bGcoor =
  \begin{bmatrix}
    \zero & \coori^T & \zero    & \zero \\
    \zero & \zero    & \coorj^T & \zero \\
    \zero & \zero    & \zero    & \coork^T \\
    \zero & \zero    & \coork^T & \coorj^T \\
    \zero & \coork^T & \zero    & \coori^T \\
    \zero & \coorj^T & \coori^T & \zero
  \end{bmatrix}.
\end{equation}

The null-space projector corresponding to the internal constraint
at the nodal level can be built by visual inspection of the Jacobian matrix as
\begin{equation}
  \nullproj =
  \begin{bmatrix}
    \bI & \zero & \zero & \zero \\
    \zero & \coorihat & \coorjhat & \coorkhat
  \end{bmatrix}^T,
\end{equation}
where the algebraic operator $\widehat{(\rdot)}$ emulates the cross product.

Since the formulation of constraints related to usual boundary conditions,
\eg rigid support, simple support, movable support \ia,
are represented by linear equations in the nodal variables and thus,
their treatment is straightforward, we omit their systematic presentation.
For further details, see \cite{Gebhardt2019a,Hente2019}.

Finally, having at hand all vectors and matrices
indicated within this subsection,
the construction of the equations corresponding to the
optimality conditions and their derivatives is straightforward.

\section{Numerical investigations}

In this section, we present three numerical examples that show the potential of the proposed approximate NLP for \DDCD. Specifically, we consider its specialization to the geometrically exact beam model. The first example presents a verification of the proposed formulation, taking the underlying three-director based standard FE formulation for geometrically exact beams that is combined with an energy-momentum time integration scheme as proposed in \cite{Gebhardt2019a, Gebhardt2019b}. We refer to that approach as EM-FEM. Such a reference numerical model is equipped with the simplest linear constitutive law. In the second example, we investigate the behavior of our framework for nonsymmetric explicitly-defined nonlinear constitutive laws, namely $\stress=\stress(\strain)$. Finally, we investigate in the third example the behavior of our framework for nonsymmetric implicitly-defined nonlinear constitutive laws, namely $\strain=\strain(\stress)$.

The problems considered along the first and second examples can also be addressed by any standard FE formulation in its dynamic setting. The problem considered along the third one can be easily addressed in the context of our approximate NLP approach, but not by standard FEM. Whether such a constitutive law is physically feasible or not will be considered in future works.

All the three examples are built on a curved beam structure whose geometry is described by a quarter of a circular arc with a total arc length of $1\,\textrm{m}$, which corresponds to a radius of $\frac{2}{\pi}\,\textrm{m}$. The (nonphysical) inertial properties are $\mathscr{E}_{00}= 10\,\textrm{Kg/m}$ and $\mathscr{E}_{11} = \mathscr{E}_{22} = 20\,\textrm{Kgm}$. The structure is uniformly discretized into $20$ two-node finite elements (\ie a total of $21$ nodes) and no further kinematic restrictions than the internal ones (orthonormality condition among the three directors) are enforced. Figure \ref{fig:resBeamStructure} shows the finite element representation and loads applied. The first node is located at the position $(0,0,0)\,\textrm{m}$. The nodes $2$--$4$ are loaded with vertical nodal forces $(0,0,-20)\,\textrm{N}$ and horizontal nodal forces $(-10,0,0)\,\textrm{N}$. For the nodes $8$--$14$ we have $(0,0,15)\,\textrm{N}$ and $(7.5,-7.5,0)\,\textrm{N}$. For the nodes $18$--$20$ we have $(0,0,-20)\,\textrm{N}$ and $(0,10,0)\,\textrm{N}$. This setting has been chosen because the action of combined spatial loads creates complex strain-stress states in space and time. Lastly, the loads are multiplied with a function that describes the variation of the external forces over time, which is defined by \eqref{eq:amplitude}, \ie~$\force=a(t)(f_1\uii+f_2\ujj+f_3\ukk)$. For all numerical examples the relative error-based tolerance of $10^{-12}$ has been set for the Newton iteration.

\begin{figure}[tp]
  \centering
  \vspace{3mm}
  \includegraphics[height=0.35\textheight]{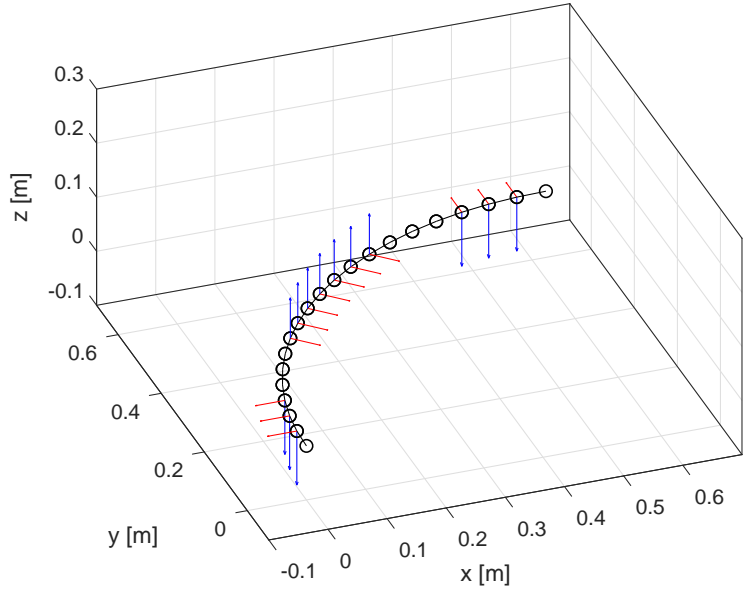}
  \caption{Finite element representation of the beam structure.
    Circles indicate nodes with internal constraints only.
    Blue and red arrows denote forces along the vertical direction $z$ and in the horizontal $x$-$y$-plane, respectively.}
  \label{fig:resBeamStructure}
\end{figure}

\begin{equation}
	a(t)=\left\{ \begin{array}{ccc}
		2t & \mathrm{for} & 0\leq t<0.5\\
		2-2t & \mathrm{for} & 0.5\leq t<1\\
		0 & \mathrm{for} & t\geq1
	\end{array}\right.
	\label{eq:amplitude}
\end{equation}

\subsection{Verification}

As indicated, we first consider the beam structure described above and the simplest linear constitutive law, defined as
\begin{equation}
(\bg^\sharp(\strainrec,\stressrec))^i =\check{s}^i-a^{ii}\check{e}_i=0
\qquad\textrm{or alternatively as}\qquad
(\bg_\flat(\strainrec,\stressrec))_i =\check{e}_i-a_{ii}\check{s}^i=0,
\end{equation}
with $a^{ii}=a^{-1}_{ii}$ and the (nonphysical) values $a^{11}=a^{22}= 75\,\mathrm{N}$, $a^{33}= 100\,\mathrm{N}$, $a^{44}=a^{55}= 100\,\mathrm{Nm^2}$ and $a^{66} = 200\,\mathrm{Nm^2}$. The weight matrix $\weight$ is defined as the identity.

\begin{table}[tp]
	\centering
	\def\-{$-$}
	\def\+{\hphantom{$-$}}
	\begin{tabular}{ccccc}
		\toprule
		      & $\coorz$ & $\coori$ & $\coorj$ & $\coork$ \\
		comp. & (m)      & (-)      & (-)      & (-)      \\
		\midrule
		$x$ & \+4.11665929 & \+0.40215049 & \+0.58161351 & \+0.70711267 \\
		$y$ & \-3.48003891 & \-0.40215496 & \-0.58161692 & \+0.70710732 \\
		$z$ & \-2.63091834 & \+0.82252438 & \-0.56873242 & \+0.00000122 \\
		\midrule
		$x$ & \+4.11667597 & \+0.40214954 & \+0.58161417 & \+0.70711267 \\
		$y$ & \-3.48005560 & \-0.40215401 & \-0.58161758 & \+0.70710733 \\
		$z$ & \-2.63093328 & \+0.82252532 & \-0.56873108 & \+0.00000123 \\
		\bottomrule
	\end{tabular}
	\caption{Verification (first example) - nodal variables of the $11^\textrm{th}$ node at $t=4$ s; EM-FEM (top) vs.\ approximate NLP (bottom).}
	\label{tab:emfemandanlpn11}
\end{table}

\newcommand\onepack[3]{%
	\begin{figure}[t]
		\centering
		\vspace{6mm}
		\includegraphics[width=0.7\textwidth]{#1_MotionSequence}
		\caption{#2 - sequence of motion.}
		\label{fig:#3}
	\end{figure}
}

\newcommand\twopackstress[3]{%
	\begin{figure}[tp]
		\centering
		\includegraphics[width=0.45\textwidth]{#1_ResultantForces}\hfill
		\includegraphics[width=0.45\textwidth]{#1_ResultantMoments}\\[2.45ex]
		\caption{#2 - resultant force per unit length at the $8^{\textrm{th}}$ element vs. time (left),
			          resultant moment per unit length at the $8^{\textrm{th}}$ element vs. time (right).}
		\label{fig:#3stress}
	\end{figure}
}

\newcommand\twopackmomentum[3]{%
	\begin{figure}[tp]
		\centering
		\includegraphics[width=0.45\textwidth]{#1_LinearMomentum}\hfill
		\includegraphics[width=0.45\textwidth]{#1_AngularMomentum}\\[2.45ex]
		\caption{#2 - linear momentum vs. time (left),
			          angular momentum vs. time (right).}
		\label{fig:#3momentum}
	\end{figure}
}

Figure \ref{fig:ex1verseq} shows a sequence of motion, where the initial configuration plotted in red is located at the topmost position. It is possible to observe that the structure undergoes large displacements and large changes of curvature. However, as the geometry of the structure and applied loads are symmetric, and the constitutive law considered is antisymmetric, the response computed is perfectly symmetric as expected. By starting each subsequent optimization problem warmly, \ie taking as initial guess the converged solution of the previous optimization problem, the approximate NLP solver requires three iterations on average to find the solution. Meanwhile, the EM-FEM solver requires four iterations on average. The EM-FEM warrants the discrete time invariance, and thus is a little bit harder to solve when compared to variational integrators, which do not render in general the discrete time invariance, see for instance \cite{Lew2004}. Table \ref{tab:emfemandanlpn11} presents the displacement and directors of the $11^\textrm{th}$ node at $t=4$ s for both formulations, namely results corresponding to the EM-FEM versus those corresponding to the approximate NLP. Table \ref{tab:emfemandanlplam} shows comparatively the stationary values for the linear and angular momenta. Both tables demonstrate that the results obtained with the approximate NLP are in excellent agreement with the results obtained with the EM-FEM. Figures \ref{fig:ex1verstress} and \ref{fig:ex1vermomentum} present several time histories for the time interval $[0,4]$ s: components of the cross sectional resultant force per unit length at the $8^\textrm{th}$ element, components of the cross sectional resultant moment per unit length at the $8^\textrm{th}$ element, components of the linear momentum, and components of the angular momentum. After an initial transient due to the presence of time varying loads, the resultant forces appear to reach a sort of oscillatory steady state. A similar observation can be made for the resultant moments. The linear and angular momenta reach, after the initial transient, stationary values that are identically preserved along the remainder of the simulation.

\onepack{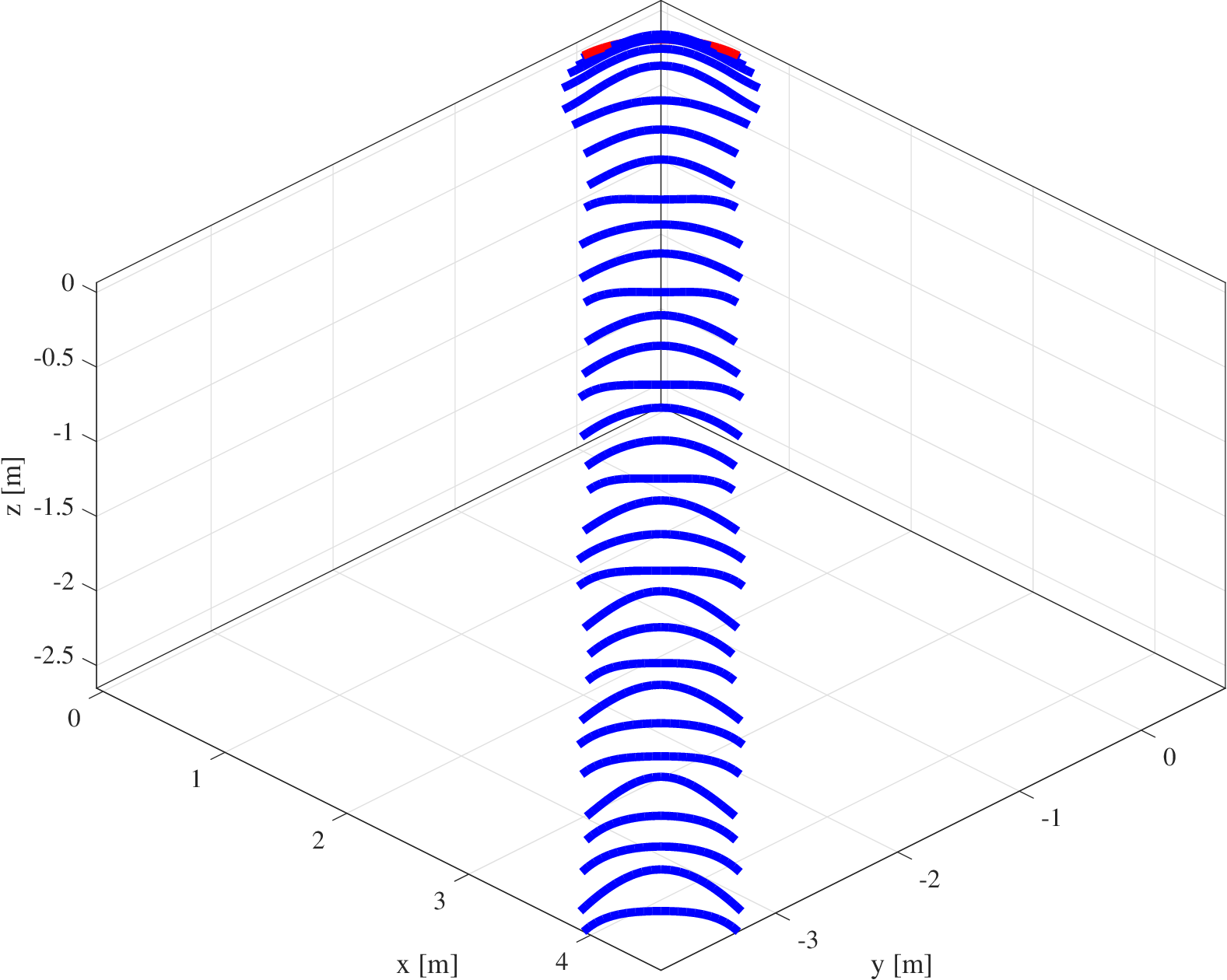}{Verification (first example)}{ex1verseq}


\twopackstress{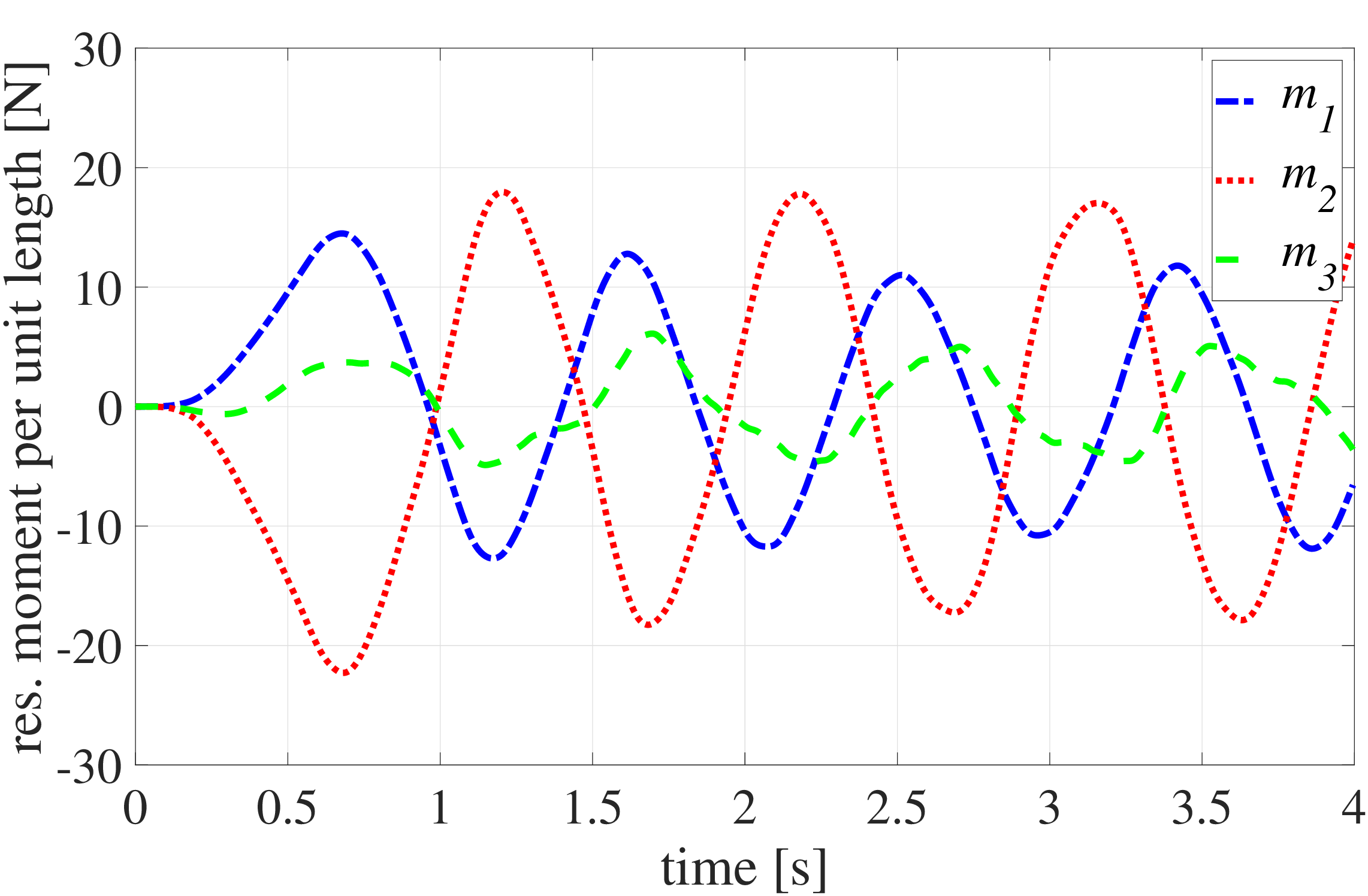}{Verification (first example)}{ex1ver}

\twopackmomentum{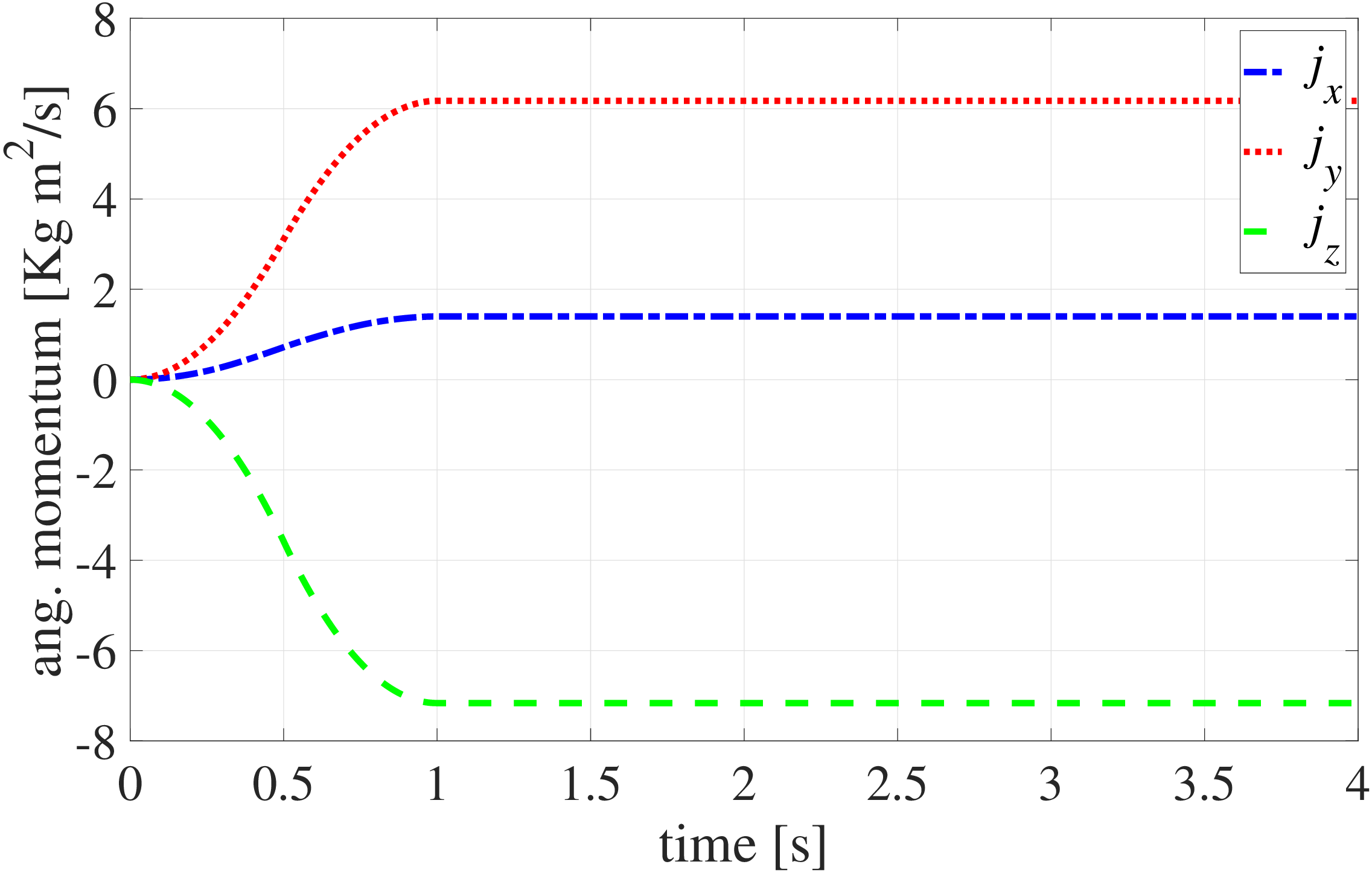}{Verification (first example)}{ex1ver}

\begin{table}[tp]
	\centering
	\def\-{$-$}
	\def\+{\hphantom{$-$}}
	\begin{tabular}{ccccc}
		\toprule
		&
		$\bm{l}$ (EM-FEM) & $\bm{l}$ (approx.~NLP) &
		$\bm{j}$ (EM-FEM) & $\bm{j}$ (approx.~NLP) \\
		comp. &
		(Kgm/s)     & (Kgm/s)     &
		(Kgm$^2$/s) & (Kgm$^2$/s) \\
		\midrule
		$x$ & \+11.25000000 & \+11.25000000 & \+1.39916065 & \+1.39915722 \\
		$y$ & \+11.25000000 & \+11.25000000 & \+6.17378053 & \+6.17377710 \\
		$z$ &\ \-7.50000000 &\ \-7.50000000 & \-7.16199129 & \-7.16199130 \\
		\bottomrule
	\end{tabular}
	\caption{Verification (first example) - linear and angular momenta; stationary values.}
	\label{tab:emfemandanlplam}
\end{table}

\subsection{Explicit stress definition}

In this second numerical example, we consider a nonlinear constitutive law given by
\begin{equation}
(\bg^\sharp(\strainrec,\stressrec))^i =\check{s}^i-a^{ii}\check{e}_i-\frac{b^{iii}}{2}(\check{e}_i)^2,
\end{equation}
where the stress resultants are defined \textit{explicitly} in terms of the strain measures.
For the coefficients $a^{ii}$, we use the same (nonphysical) values as in the verification example, and in addition we set $b^{iii} = 0.6375 a^{ii}$.

Figure \ref{fig:ex2verseq} shows a sequence of motion, in which we observe that the structure undergoes large displacements and large changes of curvature. Even, provided that the geometry of the structure and applied loads are symmetric, the response computed is, as expected due to the nonsymmetric constitutive law considered, nonsymmetric. Figures \ref{fig:ex2explicitstress} and \ref{fig:ex2explicitmomentum} present, as in the previous example, several time histories. After an initial transient due to the presence of time varying loads, the resultant forces appear to reach no steady state and show some high-frequency content. A similar comment can be made for the resultant moments. However, no high-frequency content is to be distinguished. As before, the linear and angular momenta reach, after the initial transient, stationary values that are identically preserved. Table \ref{tab:ex2explicit} presents the stationary values for the momenta. As a final comment, we can say that the approximate NLP requires normally three iterations to converge, when warmly started.

\onepack{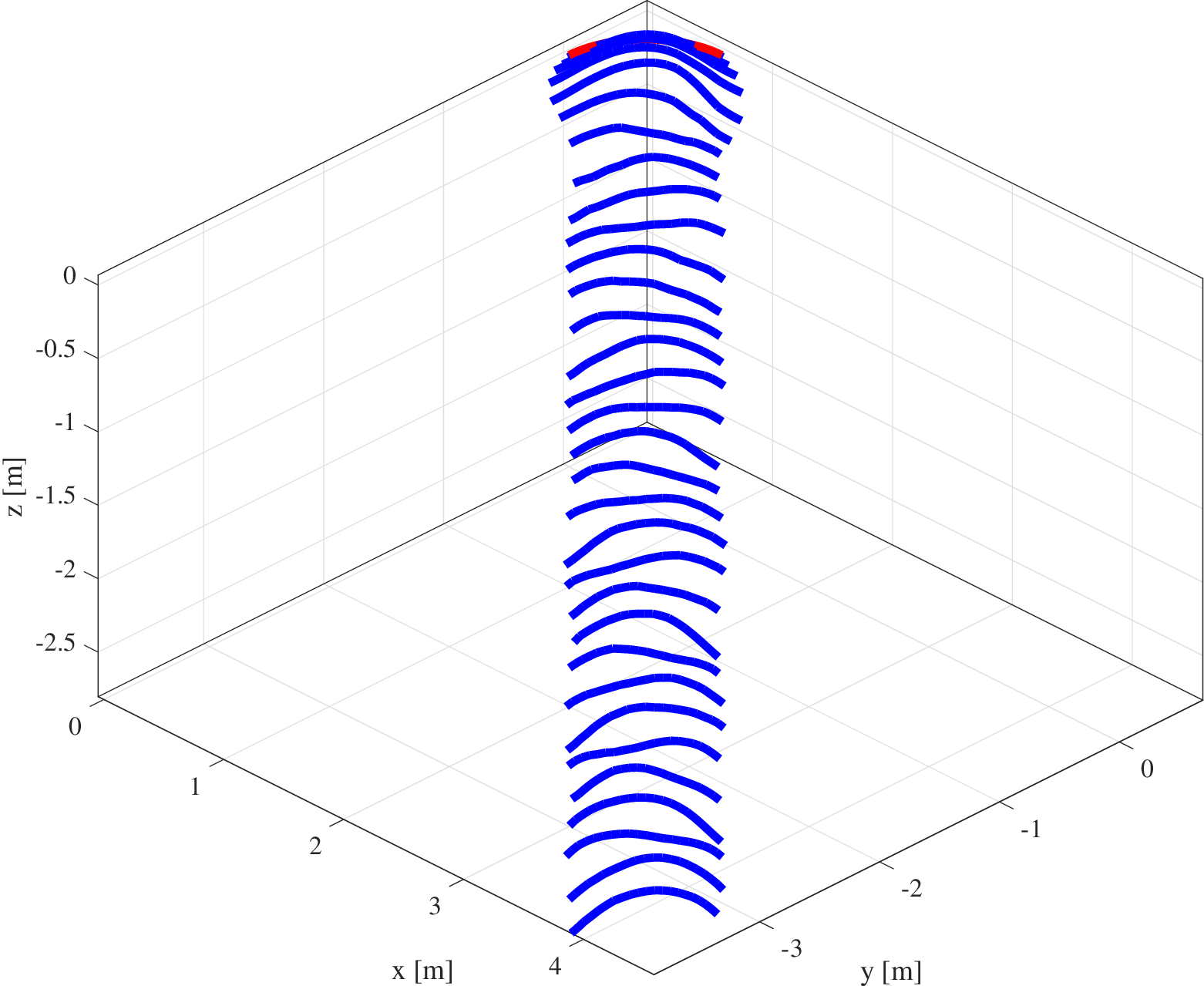}{Explicit stress definition (second example)}{ex2verseq}


\twopackstress{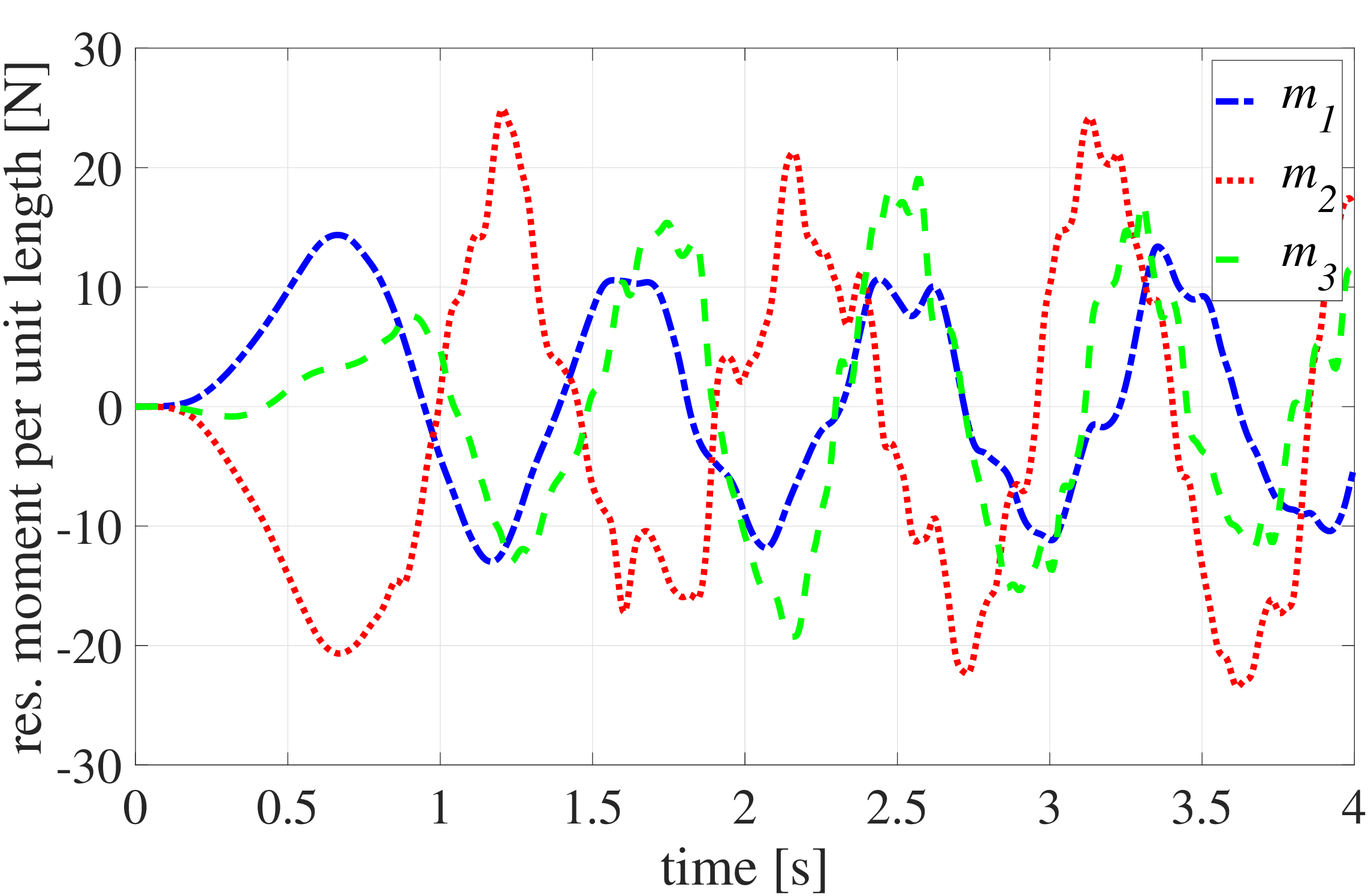}{Explicit stress definition (second example)}{ex2explicit}

\twopackmomentum{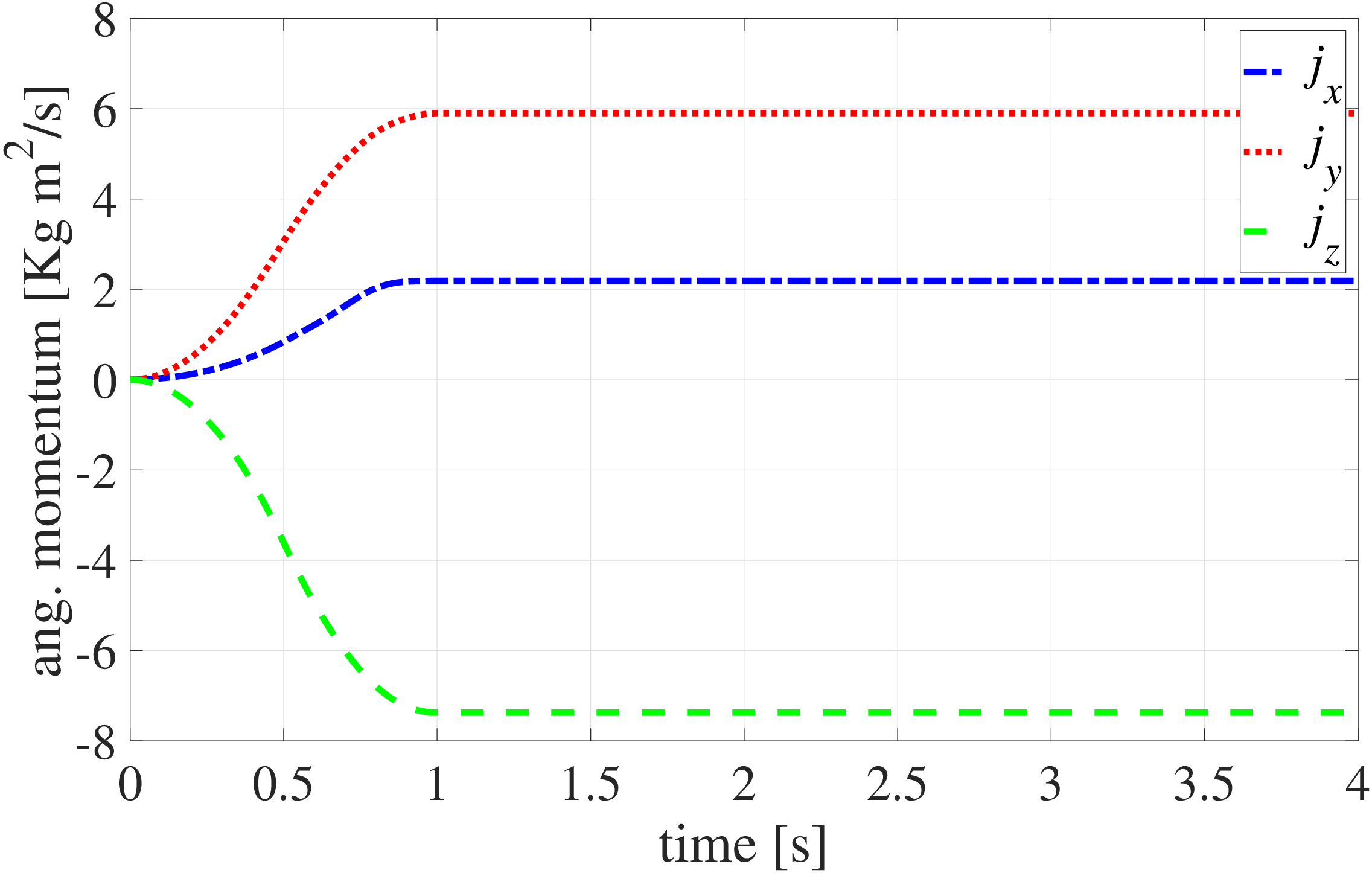}{Explicit stress definition (second example)}{ex2explicit}

\begin{table}[tp]
	\centering
	\def\-{$-$}
	\def\+{\hphantom{$-$}}
	\begin{tabular}{ccc}
		\toprule
		&
		$\bm{l}$ (approx.~NLP) &
		$\bm{j}$ (approx.~NLP) \\
		comp. &
		(Kgm/s)     &
		(Kgm$^2$/s) \\
		\midrule
		$x$ & \+11.25000000 & \+2.18791717 \\
		$y$ & \+11.25000000 & \+5.90091563 \\
		$z$ &\ \-7.50000000 & \-7.37292831 \\
		\bottomrule
	\end{tabular}
	\caption{Explicit stress definition (second example) - linear and angular momenta; stationary values.}
	\label{tab:ex2explicit}
\end{table}

\subsection{Implicit stress definition}

In this third numerical example, we consider a nonlinear constitutive law given by
\begin{equation}
(\bg_\flat(\strainrec,\stressrec))_i =\check{e}_i-a_{ii}\check{s}^i-\frac{b_{iii}}{2}(\check{s}^i)^2,
\end{equation}
where the stress resultants are defined \textit{implicitly} in terms of the strain measures.
For the coefficients $a_{ii}$, we consider the same (nonphysical) values as in the verification example, and in addition we set $b_{iii} = 0.015 a_{ii}$.

Figure \ref{fig:ex3verseq} presents a sequence of motion, in which we observe that the structure is undergoing large displacements and large changes of curvature. As in the previous case, the response is nonsymmetric due to the nonsymmetry of the constitutive law employed. Figures \ref{fig:ex3implicitstress} and \ref{fig:ex3implicitmomentum} show several time histories for the same variables that were showed in the two previous cases. After an initial transient due to the presence of time varying loads, the resultant forces do not appear to reach a steady state, and in contrast to the previous case, they show no high-frequency content. A similar observation can be made for the resultant moments, although some high-frequency content is evident. The linear and angular momenta reach, as expected, stationary values that are preserved along the simulation. Finally, Table \ref{tab:ex3implicit} presents the stationary values for the momenta. Once again, the convergence properties are excellent as before.

\onepack{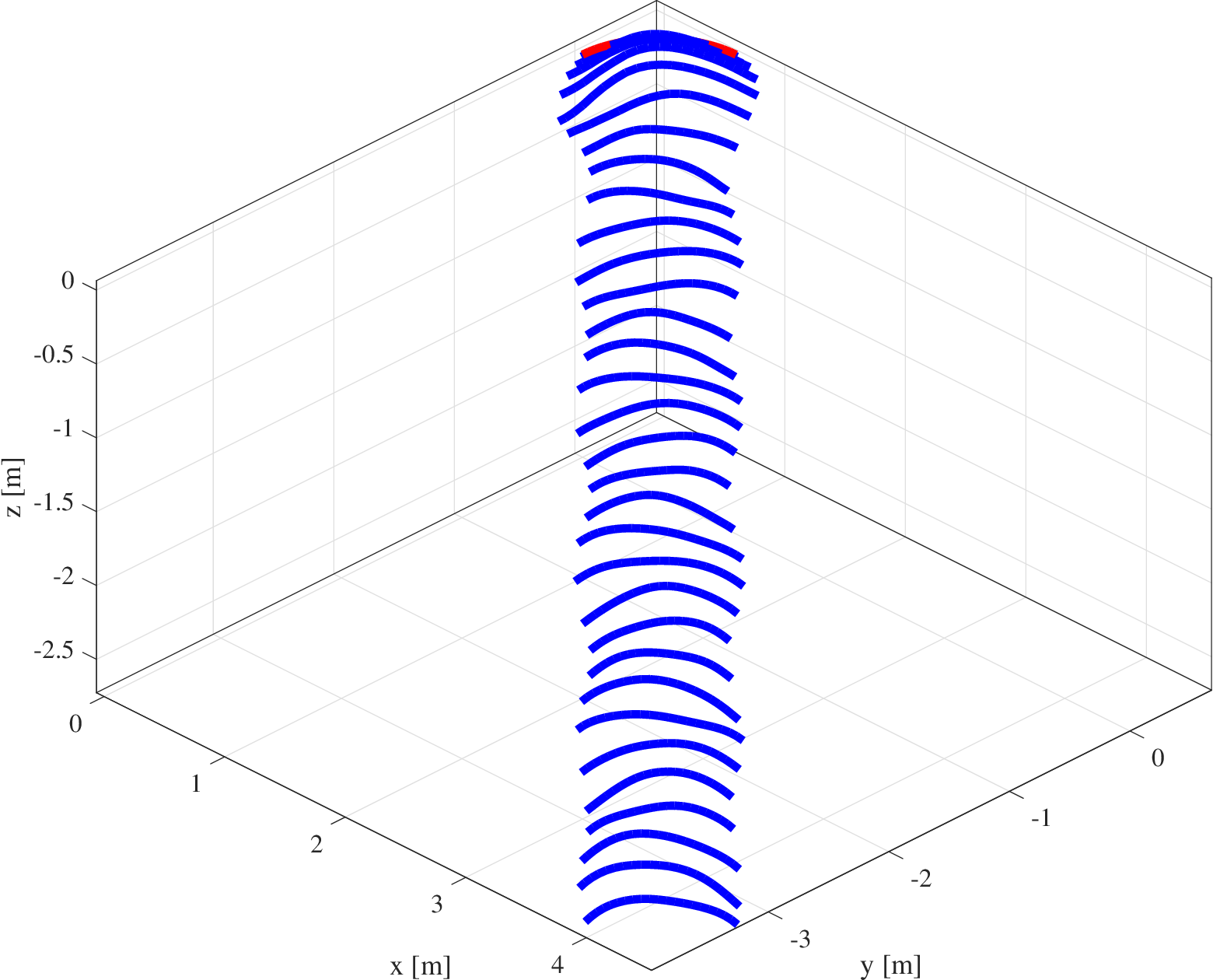}{Implicit stress definition (third example)}{ex3verseq}


\twopackstress{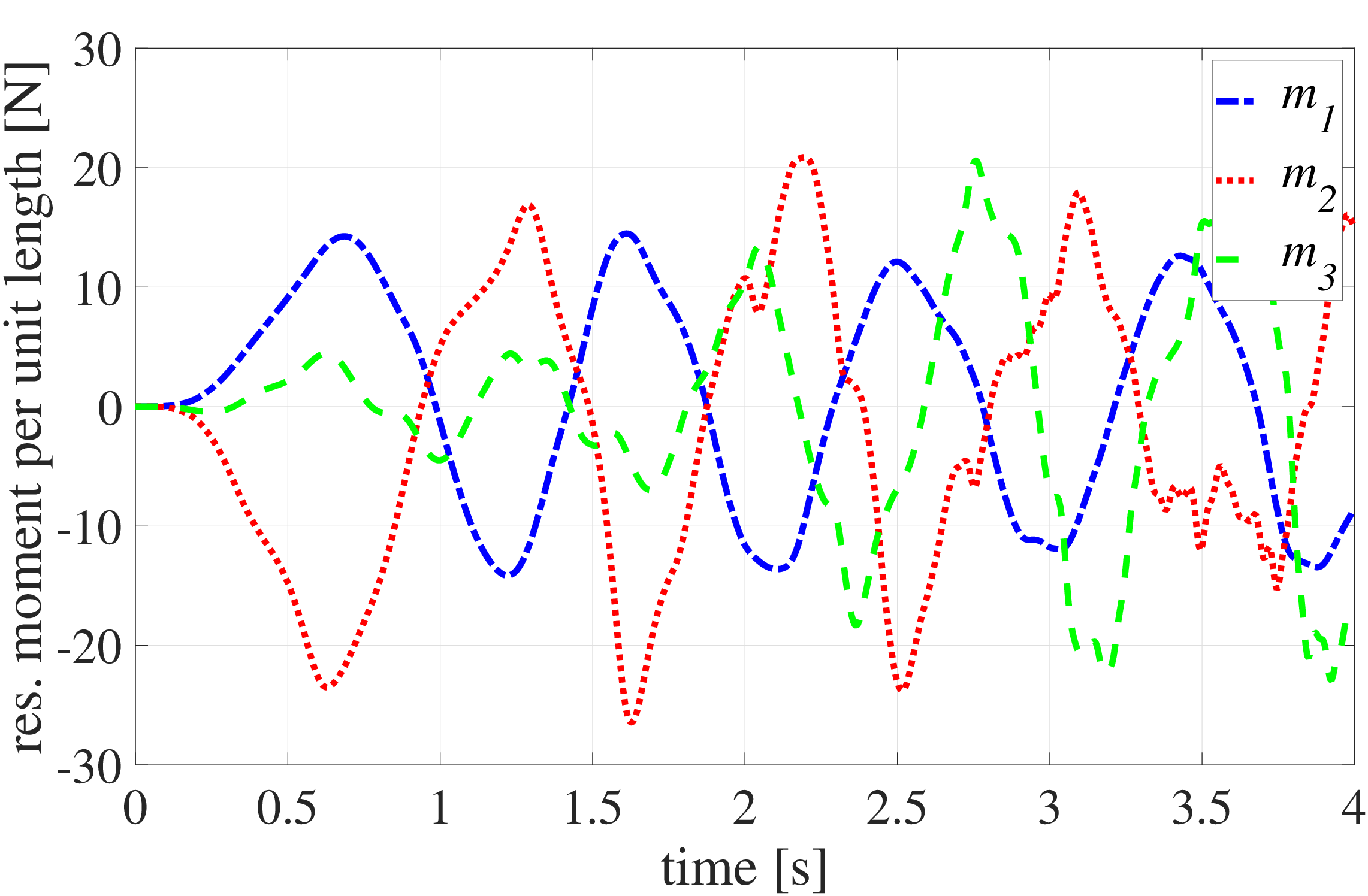}{Implicit stress definition (third example)}{ex3implicit}

\twopackmomentum{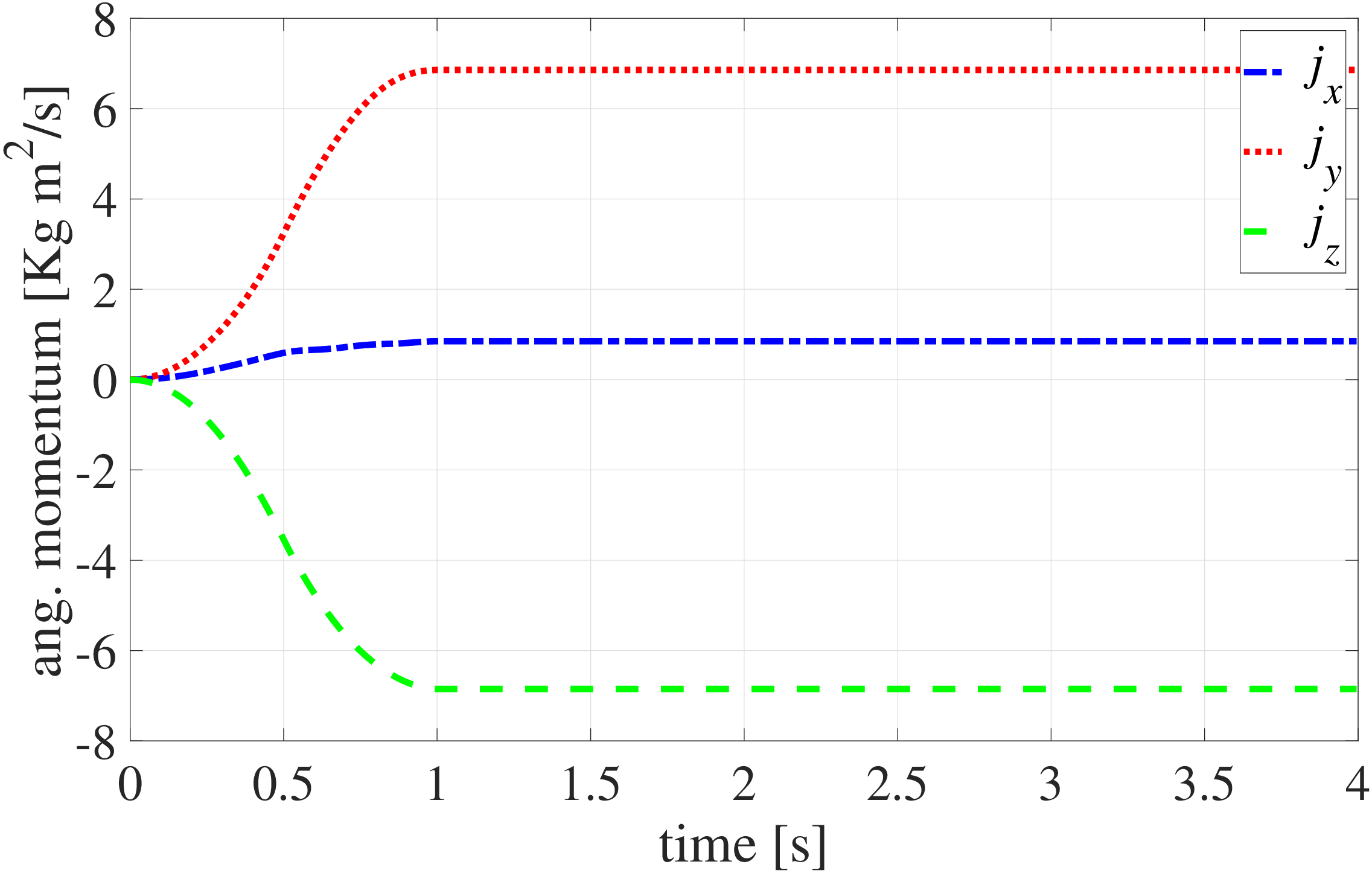}{Implicit stress definition (third example)}{ex3implicit}

\begin{table}[tp]
	\centering
	\def\-{$-$}
	\def\+{\hphantom{$-$}}
	\begin{tabular}{ccc}
		\toprule
		&
		$\bm{l}$ (approx.~NLP) &
		$\bm{j}$ (approx.~NLP) \\
		comp. &
		(Kgm/s)     &
		(Kgm$^2$/s) \\
		\midrule
		$x$ & \+11.25000000 & \+0.85015476 \\
		$y$ & \+11.25000000 & \+6.85601926 \\
		$z$ & \-7.50000000 &  \-6.84807263 \\
		\bottomrule
	\end{tabular}
	\caption{Implicit stress definition (third example) - linear and angular momenta; stationary values.}
	\label{tab:ex3implicit}
\end{table}

\section{Concluding remarks}

In this paper, we proposed
an approximate nonlinear optimization problem for \DDCD~that extends the approach recently proposed by the authors in a purely static context \cite{Gebhardt2019d}. In the dynamic setting, the resulting new formulation inherits the ability of handling: \textit{i}) kinematic constraints; and, \textit{ii}) materials whose stress-strain relationship can be implicitly approximated. Thus, the method is not limited by the requirement of some special functional structure for the definition of the material law. As the chosen dynamic setting is that inspired by a class of structure-preserving  variational integrators, it cannot be reverted to the static approach developed previously and therefore, it describes a different mathematical problem. For cases without external loads, however, the method preserves linear and angular momenta and for explicitly defined constitutive laws, it could be reverted to the discrete variational approach that preserves the underlying symplectic form. This aspect represents one main innovation of the current method.

The mathematical framework for \DDCD~was derived in detail and formulated in a self-contained fashion. To showcase the well-behaved properties of our approach, we specialized the method to the case of a geometrically exact beam formulation. We would like to emphasize that the extension of the \DDCD~paradigm to geometrically exact beam elements is another main innovation of the current work.
For the numerical examples considered in this work, the convergence behavior was observed to be excellent. Moreover, our dynamic approach fully inherits the robustness, efficiency and versatility properties of its static counterpart.

The results of the presented numerical examples clearly demonstrated that the proposed method represents a solid basis for further research. Possible research targets are the specialization to other structural models, \eg~shells and solids, the extension to solving dynamic problems in multiple time scales, the inclusion of
nonlinear optimization algorithms that can deal with inequalities and globalization techniques to warrant global convergence properties, \ia.

\section*{Acknowledgments}

\noindent C. G. Gebhardt and R. Rolfes gratefully acknowledge the financial support of the Lower Saxony Ministry of Science and Culture (research project \textit{ventus efficiens}, FKZ ZN3024) and the German Research Foundation (research project ENERGIZE, GE 2773/3-1 -- RO 706/20-1) that enabled this work. D. Schillinger acknowledges support from the German Research Foundation through the DFG Emmy Noether Award SCH 1249/2-1, and from the European Research Council via the ERC Starting Grant ``ImageToSim'' (Action No.\ 759001). \\


\bibliographystyle{ieeetr}

\bibliography{bib}

\end{document}